\documentclass{svmult}
\usepackage{etex}

\usepackage{makeidx}         
\usepackage{graphicx}        
\usepackage{multicol}        
\usepackage[bottom]{footmisc}

\makeindex             

\usepackage{amsfonts}
\usepackage{amsmath}
\usepackage{amssymb}
\usepackage{color}
\usepackage{epsfig}
\usepackage{enumerate}
\usepackage{tikz}

\usepackage[linesnumbered]{algorithm2e}

\newcommand{\comment}[1]{}

\definecolor{mydarkgreen}{rgb}{0.0,0.8,0.0}

\comment{

\newtheorem{algo}{Algorithm}[section]

}

\newcommand{\R}{\mathbb R}

\newcommand{\N}{\mathbb N}

\newcommand{\Z}{\mathbb Z}

\newcommand{\OO}{\mathcal{O}}

\newcommand{\relint}{\mathrm{relint}\;}

\newcommand{\dom}{\mathrm{dom}\,}

\newcommand{\aff}{\mathrm{aff}}
\newcommand{\conv}{\mathrm{conv}}

\newcommand{\vol}{\mathrm{vol}}

\newcommand{\inter}{\mathrm{int}\,}

\newcommand{\diam}{\mathrm{diam}}

\newcommand{\Ora}{\mathrm{O}_{\alpha,\delta}}

\begin{document}

\title*{Mirror-Descent Methods in Mixed-Integer Convex Optimization}
\author{Michel Baes, Timm Oertel, Christian Wagner, and Robert Weismantel}
\institute{Michel Baes, Timm Oertel, Christian Wagner, Robert Weismantel \at ETH Zurich, IFOR,\\
    \email{\{michel.baes, timm.oertel, christian.wagner, robert.weismantel\}@ifor.math.ethz.ch}
}
%
%
\maketitle

\abstract{In this paper, we address the problem of
  minimizing a convex function $f$ over a convex set,
  with the extra constraint that some variables must be
  integer.
  This problem, even when $f$ is a piecewise linear
  function, is NP-hard.
  We study an algorithmic approach to this problem,
  postponing its hardness to the realization of an oracle.
  If this oracle can be realized in polynomial time, then
  the problem can be solved in polynomial time as well.
  For problems with two integer variables, we show that
  the oracle can be implemented efficiently, that is, in
  $\OO(\ln(B))$ approximate minimizations of $f$ over the continuous
  variables, where $B$ is a known bound on the absolute
  value of the integer variables.
  Our algorithm can be adapted to find the second best point
  of a purely integer convex optimization problem in two dimensions, and
  more generally its $k$-th best point.
  This observation allows us to formulate a finite-time
  algorithm for mixed-integer convex optimization.}

\section{Introduction}
One of the highlights in the list of publications of Martin Gr\"otschel is his joint book with L\'aszl\'o Lov\'asz and Alexander Schrijver on Geometric Algorithms and Combinatorial Optimization \cite{GroetschelLovaszSchrijver-Book88}. This book develops a beautiful and general theory of optimization over (integer) points in convex sets. The generality comes from the fact that the  convex sets under consideration are presented by oracles (membership, separation in different variations, optimization).
The algorithms and their efficiency typically depend on the oracle presentation of the underlying convex set.
This is precisely the theme of this paper as well:
we present an algorithmic framework for solving mixed-integer convex optimization problems that is based on an
oracle.
Whenever the oracle can be realized efficiently, then the overall running time of the
optimization algorithm is efficient as well.

One of the results from the book
\cite{GroetschelLovaszSchrijver-Book88} that is perhaps
closest to our results is the following.
Here and throughout the paper $B(p,r)$ denotes a ball of
radius $r$ with center $p$.

\begin{theorem}\label{thm:GLS}
\cite[Theorem~6.7.10]{GroetschelLovaszSchrijver-Book88}
Let $n$ be a fixed integer and $K\subseteq\R^n$ be any convex set given by a weak separation oracle and for which there exist $r,R>0$ and $p\in K$ with $B(p,r) \subseteq K \subseteq B(0,R)$. There exists an oracle-polynomial algorithm that, for every fixed $\epsilon > 0$, either finds an integral point in $K+B(0,\epsilon)$ or concludes that $K \cap \Z^n = \emptyset$.
\end{theorem}

The main distinction between results presented here and results from \cite{GroetschelLovaszSchrijver-Book88} of such flavor as Theorem \ref{thm:GLS} results from dropping the assumption that we know  a ball $B(p,r) \subseteq K$.
At first glance, this might sound harmless, but it is not since the proof of Theorem \ref{thm:GLS} in \cite{GroetschelLovaszSchrijver-Book88} uses a combination of the ellipsoid algorithm \cite{Khachiyan} and a Lenstra-type algorithm \cite{Lenstra83}.
In fact, dropping the assumption to know  a ball $B(p,r) \subseteq K$  requires new algorithmic frameworks.

Let us now make precise our assumptions.
We study a general mixed-integer convex optimization problem
of the kind
\begin{equation}\label{eq:general_MIP_f}
 \min\{f(\hat x,y): \left(\hat x,y\right) \in S\cap(\Z^n\times\R^d)\},
\end{equation}
where the function $f:\R^{n+d}\to\R_+\cup\{+\infty\}$ is a
nonnegative proper convex function, i.e., there is a point
$z\in\R^{n+d}$ with $f(z)<+\infty$.
Moreover, $S\subseteq\R^{n+d}$ is a convex set that is defined by a finite number of convex functional
constraints, i.e.,
$S := \{(x,y)\in \R^{n+d} : g_i(x,y) \le 0 \text{
  for } 1 \le i \le m\}$.
We denote by $\langle\cdot,\cdot\rangle$ a scalar product.
The functions $g_i: \R^{n+d} \to \R$ are
differentiable convex functions and encoded by a so-called
\emph{first-order oracle}. Given any point $(x_0,y_0) \in
\R^{n+d}$, this oracle returns, for every $i \in \{1,
\dots, m\}$, the function value $g_i(x_0,y_0)$ together with a subgradient $g'_i(x_0,y_0)$, that is, a vector satisfying:
\[
g_i(x,y) - g_i(x_0,y_0) \geq \langle g'_i(x_0,y_0),\left(x-x_0,y-y_0\right)\rangle  
\]
for all $(x,y) \in \R^{n+d}$.

In this general setting, very few algorithmic frameworks exist. The most commonly used one is ``outer approximation'', originally proposed in \cite{DuranGrossmann86} and later on refined in \cite{ViswanathanGrossmann90, FletcherLeyffer94, Bonamietal2008}.
This scheme is known to be finitely converging, yet there is no analysis regarding the number of iterations it takes to solve problem \eqref{eq:general_MIP_f} up to a certain given accuracy.

In this paper we present oracle-polynomial algorithmic
schemes that are (i) amenable to an analysis and (ii) finite
for any mixed-integer convex optimization problem.
Our schemes also give rise to the fastest algorithm so far for
solving mixed-integer convex optimization problems in
variable dimension with at most two integer variables.

\section{An algorithm based on an ``improvement oracle''}

\comment{
The generic optimization problem \eqref{eq:general_MIP} is slightly more general to the following formulation, which we will use throughout the paper:

\begin{equation}\label{eq:general_MIP_f}
 \min\{f(\hat x,y): \left(\hat x,y\right) \in S\cap(\Z^n\times\R^d)\},
\end{equation}

where the function $f:\R^{n+d}\to\R_+\cup\{+\infty\}$ is a nonnegative proper\footnote{A function $f:\R^{n+d}\to\R_+\cup\{+\infty\}$ is \emph{proper} if there is a point $z\in\R^{n+d}$ where $f(z)<+\infty$.} convex
function and $S\subseteq\R^{n+d}$ is a convex set that can
be defined by a finite number of convex functional
constraints. The problem class \eqref{eq:general_MIP}
contains the class \eqref{eq:general_MIP_f}: since $S$ can
be written as $S:=\{(x,y)\in\R^{n+d}:g_i(x,y)\leq 0\textrm{
  for }1\leq i\leq m\}$, the objective of problem
\eqref{eq:general_MIP_f} can be made linear by including the
convex constraint $f(x,y)\leq t$ while the objective
consists in minimizing the continuous variable $t$. On the
other hand, if \eqref{eq:general_MIP} has a known upper
bound to its optimum, we can replace its linear objective
function by an affine one that is negative on $S$, and take
for $f$ its opposite to get a problem in the class
\eqref{eq:general_MIP_f}.
}

We study in this paper an algorithmic approach to solve
\eqref{eq:general_MIP_f}, postponing its hardness to the
realization of an improvement oracle defined below.
If this oracle can be realized in polynomial time, then the
problem can be solved in polynomial time as well.
An oracle of this type has already been used in a number of
algorithms in other contexts, such as in
\cite{Arora_Kale_07} for semidefinite problems.

\begin{definition}[Improvement Oracle]
Let $\alpha,\delta \ge 0$.
For every $z \in S$, the oracle
\begin{enumerate}
  \item[\textbf{a.}] returns
    $\hat{z} \in S \cap (\Z^n \times \R^d)$ such that
    $f(\hat{z}) \le (1+\alpha) f(z) + \delta$, and/or
  \item[\textbf{b.}] asserts correctly that there is no
    point $\hat{z} \in S \cap (\Z^n \times \R^d)$ for which
    $f(\hat{z}) \le f(z)$.
\end{enumerate}
We denote the query to this oracle at $z$ by
$\mathrm{O}_{\alpha,\delta}(z)$.
\end{definition}

As stressed in the above definition, the oracle might content itself with a feasible point $\hat z$ satisfying the inequality in \textbf{a} without addressing the problem in \textbf{b}. However, we do not exclude the possibility of having an oracle that can occasionally report both facts.
In that case, the point $\hat z$ that it outputs for the input point $z\in S$ must satisfy:
\[
f(\hat z)-\hat f^*\leq \alpha f(z) + \delta + (f(z)-\hat f^*)\leq \alpha f(z) + \delta \leq \alpha \hat f^* +
\delta,
\]
where $\hat f^*$ is the optimal objective value of \eqref{eq:general_MIP_f}.
Thus $f(\hat z)\leq (1+\alpha) f(z) + \delta$, and it is not possible to hope for a better point of $S$ from the oracle. We can therefore interrupt the computations and output $\hat z$ as the final result of our method.

In the case where $\hat f^*>0$ and $\delta=0$, the improvement oracle
might be realized by a relaxation of the problem of finding a suitable
$\hat z$: in numerous cases, these relaxations come with a
guaranteed value of $\alpha$. In general, the realization of
this oracle might need to solve a problem as difficult as
the original mixed-integer convex instance, especially when $\alpha=\delta=0$. Nevertheless, we will point out several situations where this oracle can actually be realized quite efficiently, even with $\alpha=0$.

The domain of $f$, denoted by $\dom f$, is the set of all the points $z\in\R^{n+d}$ with $f(z)<+\infty$. For all $z\in\dom f$, we denote by $f'(z)$ an element of the subdifferential $\partial f(z)$ of $f$. We represent by $\hat z^* = (\hat x^*,y^*)$ a minimizer of \eqref{eq:general_MIP_f}, and set $\hat f^*:=f(\hat z^*)$; more generally, we use a hat ($\hat\cdot$) to designate vectors that have their $n$ first components integral by definition or by construction.

Let us describe an elementary method for solving \emph{Lipschitz continuous} convex problems on $S$ approximately. Lipschitz continuity of $f$ on $S$, an assumption we make from now on, entails that, given a norm
$||\cdot||$ on $\R^{n+d}$, there exists a constant $L>0$ for which:
\[
|f(z_1)-f(z_2)|\leq L||z_1-z_2||
\]
for every $z_1,z_2\in S$.
Equivalently, if $||\cdot||_*$ is the dual norm of
$||\cdot||$, we have $||f'(z)||_*\leq L$ for every
$f'(z)\in\partial f(z)$ and every $z\in\dom f$.

Our first algorithm is a variant of the well-known
Mirror-Descent Method (see Chapter~3 of
\cite{Nemirovski_Yudin_book}).
It requires a termination procedure, which used alone
constitutes our second algorithm as a minimization algorithm
on its own.
However, the second algorithm requires as input an
information that is a priori not obvious to get:
a point $z \in S$ for which $f(z)$ is a (strictly)
positive lower bound of $\hat f^*$.

Let $V:\R^{n+d}\to\R_+$ be a differentiable $\sigma$-strongly convex function with respect to the norm $||\cdot||$, i.e., there exists a $\sigma>0$ for which, for every $z_1,z_2\in\R^{n+d}$, we have:
\[
V(z_2)-V(z_1)-\langle V'(z_1),z_2-z_1\rangle\geq \frac{\sigma}{2}||z_2-z_1||^2.
\] 
We also use the conjugate $V_*$ of $V$ defined by $V_*(s):=\sup\{\langle s,z\rangle-V(z):z\in\R^{n+d}\}$ for every $s\in\R^{n+d}$. We fix $z_0\in S$ as the starting point of our algorithm and denote by 
$M$ an upper bound of $V(\hat z^*)$. We assume that the solution of the problem $\sup\{\langle s,z\rangle-V(z):z\in\R^{n+d}\}$ exists and can be computed easily, as well as the function $\rho(w):=\min\{||w-z||:z\in S\}$ for every $w\in\R^{n+d}$, its subgradient, and the minimizer $\pi(w)$. In an alternative version of the algorithm we are about to describe, we can merely assume that the problem $\max\{\langle s,z\rangle-V(z):z\in S\}$ can be solved efficiently.

A possible building block for constructing an algorithm to
solve \eqref{eq:general_MIP_f} is the continuous optimum of
the problem, that is, the minimizer of
\eqref{eq:general_MIP_f} without the integrality
constraints.
The following algorithm is essentially a standard procedure
meant to compute an approximation of this continuous
minimizer, lined with our oracle that constructs
simultaneously a sequence of mixed-integer feasible points
following the decrease of $f$.
Except in the rare case when we produce a provably suitable
solution to our problem, this algorithm provides a point
$z\in S$ such that $f(z)$ is a lower bound of $\hat f^*$.
Would this lower bound be readily available, we can jump
immediately to the termination procedure
(see Algorithm~\ref{algo:termination}).

\LinesNumberedHidden
\begin{algorithm}[htbp]
\DontPrintSemicolon
\KwData{$z_0\in S$.}
Set $\hat z_0:=z_0$, $w_0:=z_0$, $s_0:=0$, and $\hat f_0:=f(\hat z_0)$.\;
Select sequences $\{h_k\}_{k\geq 0}$, $\{\alpha_k\}_{k\geq 0}$, $\{\delta_k\}_{k\geq 0}$.\;
\For{$k=0,\dots,N$}{
	Compute $f'(z_k)\in\partial f(z_k)$ and $\rho'(w_k)\in\partial\rho(w_k)$.\;
	Set $s_{k+1}:=s_k-h_k f'(z_k)-h_k ||f'(z_k)||_*\rho'(w_k)$.\;
	Set $w_{k+1}:=\arg\max\{\langle s_{k+1},z\rangle-V(z):z\in\R^{n+d}\}$.\;
	Set $z_{k+1}:=\arg\min\{||w_{k+1}-z||:z\in S\}$.\;
	Compute $f(z_{k+1})$.\;
	\lIf{$f(z_{k+1})\geq \hat f_k$}{$\hat z_{k+1}:=\hat z_k$, $\hat f_{k+1}:=\hat f_k$.\;}
	\Else{
		Run $\mathrm{O}_{\alpha_{k+1},\delta_{k+1}}(z_{k+1})$.\;
		\uIf{the oracle reports {\bf a} and {\bf b}}{
			Terminate the algorithm and return the oracle output from \textbf{a}.\;
			}
		\uElseIf{ the oracle reports {\bf a} but not {\bf b}}{
			Set $\hat z_{k+1}$ as the oracle output and $\hat f_{k+1}:=\min\{f(\hat z_{k+1}),\hat f_k\}$.\;
			}
		\Else{
			Run the termination procedure with $z_0:=z_{k+1}$, $\hat z_0:=\hat z_{k+1}$,\;
			return its output, and terminate the algorithm.\;
			}
		}
	}
\caption{Mirror-Descent Method.\label{algo:mirror}} 
\end{algorithm} 

\begin{algorithm}[htbp]
\DontPrintSemicolon
\KwData{$z_0\in S$ with $f(z_0)\leq \hat f^*$,  $\hat z_0\in S\cap(\Z^n\times\R^d)$.}
Set $l_0:= f(z_0)$, $u_0:=f(\hat z_0)$.\;
Choose $\alpha,\delta\geq 0$. Choose a subproblem accuracy $\epsilon'>0$.\;
\For{$k\geq 0$}{
	Compute using a bisection method a point $z_{k+1} =\lambda z_k + (1-\lambda)\hat z_k$\;
	for $0\leq\lambda\leq 1$, for which $f(z_{k+1})-(l_k(\alpha+1)+u_k)/(\alpha+2)\in [-\epsilon',\epsilon']$.\;
	Run $\mathrm{O}_{\alpha,\delta} (z_{k+1})$.\;
	\uIf{the oracle reports {\bf a} and {\bf b}}{
		Terminate the algorithm and return the oracle output from \textbf{a}.\;
		}
	\uElseIf{ the oracle reports {\bf a} but not {\bf b}}{
		Set $\hat z_{k+1}$ as the oracle output,
                $l_{k+1}:=l_k$,
                $u_{k+1}:=\min\{f(\hat z_{k+1}),u_k\}$.\;
		}
	\Else{
		Set $\hat z_{k+1}:=\hat z_k$,
                $l_{k+1}:=f(z_{k+1})$,
                $u_{k+1}:=u_k$.\;
		}
	}
\caption{Termination procedure.\label{algo:termination}} 
\end{algorithm}

The following proposition is an extension of the standard proof of convergence of Mirror-Descent Methods. We include it here for the sake of completeness.

\begin{proposition}
Suppose that the oracle reports $\bf{a}$ for $k =
0,\ldots,N$ in Algorithm \ref{algo:mirror}, that is, it delivers an output $\hat z_k$ for every
iteration $k = 0,\ldots,N$. Then:
\[
\frac{1}{\sum_{k=0}^Nh_k}\sum_{k=0}^N\frac{h_kf(\hat
z_k)}{1+\alpha_k} - f(\hat z^*) \leq\frac{M}{\sum_{k=0}^Nh_k} +
\frac{2L^2}{\sigma}\cdot\frac{\sum_{k=0}^Nh_k^2}{\sum_{k=0}^Nh_k} +
\frac{1}{\sum_{k=0}^Nh_k}\sum_{k=0}^N\frac{h_k\delta_k}{1+\alpha_k}.
\]
\end{proposition}

\begin{proof}
\smartqed
Since $V$ is $\sigma$-strongly convex with
respect to the norm $||\cdot||$, its conjugate $V_*$ is
differentiable and has a Lipschitz continuous gradient of
constant $1/\sigma$ for the norm
$||\cdot||_*$, i.e.,
$V_\ast(y) - V_\ast(x) \le \langle V_\ast'(x),
  y-x \rangle + \frac{1}{2 \sigma} \|y-x\|_\ast^2$
(see~\cite[Chapter~X]{HiriartUrrutyLemarechal93b}).
Also $w_k = V_*'(s_k)$, in view of \cite[Theorem~23.5]{Rock_subgradients}.
Finally, for every $z\in S$, we can write
$\rho(w_k) + \langle\rho'(w_k),z-w_k\rangle\leq\rho(z)=0$.
Thus:
\begin{equation}\label{eq:interm_rho}
\langle\rho'(w_k),w_k-\hat z^*\rangle\geq \rho(w_k) = ||\pi(w_k)-w_k|| = ||z_k-w_k||.
\end{equation}
Also, $||\rho'(w_k)||_*\leq 1$, because for every $z\in\R^{n+d}$:
\begin{eqnarray}
\langle\rho'(w_k),z-w_k\rangle &\leq& \rho(z)-\rho(w_k) = ||z-\pi(z)||-||w_k-\pi(w_k)|| \nonumber\\
   &\leq& ||z-\pi(w_k)||-||w_k-\pi(w_k)|| \leq ||z-w_k||.\label{eq:rho<=1}
\end{eqnarray}
By setting $\phi_k:=V_*(s_k) - \langle s_k,\hat z^*\rangle$,
we can write successively for all $k \ge 0$:
{\allowdisplaybreaks
\begin{eqnarray*}
\phi_{k+1} & = & V_*(s_{k+1}) -
                \langle s_{k+1},\hat z^*\rangle\\
& \le & V_*(s_k) + \langle V_\ast'(s_k), s_{k+1}-s_k \rangle
        + \frac{1}{2 \sigma} \|s_{k+1}-s_k\|_\ast^2 -
        \langle s_{k+1},\hat z^*\rangle.\\
& = & (V_*(s_k) - \langle s_k,\hat z^*\rangle) + \left
      \langle V'_*(s_k)-\hat z^*,s_{k+1}-s_k \right \rangle
      + \frac{1}{2\sigma}\left\|s_{k+1}-s_k\right\|^2_*\\
& = & \phi_k - h_k \left \langle w_k-z_k,f'(z_k) \right
      \rangle + h_k \left \langle \hat z^* - z_k,f'(z_k)
      \right \rangle\\
& & - h_k||f'(z_k)||_*\left\langle w_k-\hat z^*,\rho'(w_k)
    \right \rangle +
    \frac{h_k^2 \|f'(z_k)\|_\ast^2}{2\sigma} \left
    \|\frac{f'(z_k)}{\|f'(z_k)\|_\ast} + \rho'(w_k)
    \right\|^2_*,
\end{eqnarray*}
}where the inequality follows from the Lipschitz continuity
of the gradient of $V_\ast$, and the last equality from the
identities
$V'_\ast(s_k) = w_k$, $s_{k+1} - s_k = -h_k f'(z_k) -
h_k \| f'(z_k)\|_\ast \rho'(w_k)$, and
$V_*(s_k) - \langle s_k,\hat z^* \rangle = \phi_k$.
By the definition of the dual norm, it holds
$-h_k \langle w_k-z_k,f'(z_k) \rangle \le
h_k \|f'(z_k)\|_\ast \|w_k-z_k\|$.
Moreover, convexity of $f$ implies
$h_k \langle \hat z^* - z_k,f'(z_k) \rangle \le
f(\hat z^\ast) - f(z_k)$.
Using this in the above expression we get:
{\allowdisplaybreaks
\begin{eqnarray*}
\phi_{k+1} & \le & \phi_k  + h_k ||f'(z_k)||_* \left
  (||w_k-z_k|| - \left \langle w_k - \hat z^*,\rho'(w_k)
  \right \rangle \right) \\
& & + h_k(f(\hat z^*) - f(z_k)) +
  \frac{h_k^2||f'(z_k)||_*^2}{2\sigma} \left( \left
  \|\frac{f'(z_k)}{||f'(z_k)||_*} \right \|_\ast +
  \left \|\rho'(w_k)\right\|_* \right)^2\\
& \le & \phi_k + h_k(f(\hat z^*)-f(z_k)) +
  \frac{2h_k^2||f'(z_k)||_*^2}{\sigma}\\
& \le & \phi_k + h_k \left( f(\hat z^*) -
  \frac{f(\hat z_k)-\delta_k}{1+\alpha_k} \right) +
  \frac{2h_k^2||f'(z_k)||_*^2}{\sigma},
\end{eqnarray*}
}where the second inequality follows from
\eqref{eq:interm_rho} and $\|\rho'(w_k)\|_* \le 1$, and the
third inequality from the fact that the oracle reports
${\bf a}$.
Summing up the above inequalities from $k:=0$ to $k:=N$ and
rearranging, it follows:
\[
\frac{1}{\sum_{k=0}^Nh_k }\sum_{k=0}^N\frac{h_k(f(\hat
z_k)-\delta_k)}{1+\alpha_k} - f(\hat z^*) \leq\frac{\phi_0 -
\phi_{N+1}}{\sum_{k=0}^Nh_k} +
\frac{2\sum_{k=0}^Nh_k^2||f'(z_k)||^2_*}{\sigma\sum_{k=0}^Nh_k}.
\]
Note that $||f'(z_k)||_*\leq L$, $\phi_0 =
\sup\{-V(z):z\in\R^{n+d}\}\leq 0$, and $ \phi_{N+1}\geq -V(\hat z^*)\geq -M$, yielding the
desired result. \qed 
\end{proof}

In the special case when $\alpha_k = \alpha$ and
$\delta_k = \delta$ for every $k\geq 0$, we can
significantly simplify the above results.
According to the previous proposition, we know that:
\[\left(\sum_{k=0}^Nh_k\right)\left(\frac{\hat f_N-\delta}{1+\alpha} - \hat f^*\right) = \left(\sum_{k=0}^Nh_k\right)\left(\frac{\min_{1\leq i\leq N}f(\hat z_i)-\delta}{1+\alpha} - \hat f^*\right) \]
\begin{equation}\label{eq:bound_mirror_interm}
 \leq \sum_{k=0}^N\frac{h_k(f(\hat z_k)-\delta)}{1+\alpha} -
\left(\sum_{k=0}^Nh_k\right)\hat f^* \leq M +
\frac{2L^2}{\sigma}\sum_{k=0}^Nh_k^2.
\end{equation}
We can divide both sides of the above inequality by
$\sum_{k=0}^Nh_k$, then determine the step-sizes
$\{h_k:0\leq k\leq N\}$ for which the right-hand
  side is minimized.
However, with this strategy, $h_0$ would depend on $N$, which is a priori unknown at the first iteration. Instead, as in \cite{Nesterov-Book04}, we use a step-size of the form $h_k = c/\sqrt{k+1}$ for an appropriate constant $c>0$, independent of $N$. Note that:
\[
\sum_{k=0}^N\frac{1}{k+1} =
\sum_{k=1}^{N+1}\frac{1}{k}\leq\int_1^{N+1}\frac{dt}{t} + 1
=\ln(N+1)+1.
\]
If we choose $c:=\sqrt{\frac{\sigma M}{2L^2}}$, the right-hand side of \eqref{eq:bound_mirror_interm} can be upper-bounded by $M\ln(N+1)+2M$. Finally, since
\[
\frac{1}{c}\sum_{k=0}^Nh_k = \sum_{k=0}^N\frac{1}{\sqrt{k+1}} =
\sum_{k=1}^{N+1}\frac{1}{\sqrt{k}}\geq\int_1^{N+2}\frac{dt}{\sqrt{t}}
= 2\sqrt{N+2}-2,
\]
we can thereby conclude that:
\begin{equation}\label{eq:complexity_mirror}
\frac{\hat f_N-(1+\alpha)\hat f^*-\delta}{1+\alpha}\leq
L\sqrt{\frac{M}{2\sigma}}\cdot\frac{\ln(N+1)+2}{\sqrt{N+2}-1}.
\end{equation}
As the right-hand side converges to $0$ when $N$
  goes to infinity, Algorithm~\ref{algo:mirror} converges
to an acceptable approximate solution or calls the
termination procedure.

Let us now turn our attention to the termination procedure.
We assume here that the oracle achieves a constant quality,
that is, that there exists $\alpha,\delta\geq 0$ for which
$\alpha_k = \alpha$ and $\delta_k = \delta$ for every
$k \ge 0$. 
\begin{proposition}\label{prop:termination_complexity}
Assume that $f(\hat z_0)\geq f(z_0)>0$, and that there is no point
$\hat z\in S\cap(\Z^n\times\R^d)$ for which $f(z_0)>f(\hat
z)$.
\begin{enumerate}[(a)]
\item \label{accuracy} The termination procedure cannot guarantee an accuracy
  better than: 
\begin{equation}\label{eq:best possible accuracy for termination}
f(\hat z)\leq \hat f^* + (2+\alpha) \left(\alpha \hat f^* + (1+\alpha)\epsilon'+ \delta\right).
\end{equation}
\item \label{steps} For every $\epsilon > 0$, the termination
procedure finds a point $\hat z\in S\cap(\Z^n\times\R^d)$ satisfying:
\[
f(\hat z)-\hat f^*\leq \epsilon \hat f^* + (2+\alpha)
\left(\alpha\hat f^* + (1+\alpha)\epsilon'+ \delta\right)
\]
within
\[
\max\left\{\left\lceil\ln\left.\left(\frac{f(\hat
z_0)-f(z_0)}{f(z_0)\epsilon}\right)\right/\ln\left(\frac{2+\alpha}{1+\alpha}\right)\right\rceil,0\right\}
\]
iterations.
\end{enumerate}
\end{proposition}
\begin{proof}
\smartqed
{\it Part (a).} At every iteration $k$, there is by
construction no $\hat z\in S\cap(\Z^n\times\R^d)$ for which
$l_k>f(\hat z)$.
Also, $f(\hat z_k) \ge  u_k \geq\hat f^*$.
For convenience, we denote $(1+\alpha)/(2+\alpha)$ by
$\lambda$ in this proof, and we set $\Delta_k:=u_k-l_k$ for every $k\geq 0$.

Suppose first that the oracle finds a new point $\hat
z_{k+1}\in S\cap(\Z^n\times\R^d)$ at iteration $k$. Then:
\[
f(\hat z_{k+1})\leq(1+\alpha)f(z_{k+1})+\delta\leq(1+\alpha)
\left(\lambda l_k+(1-\lambda)u_k+\epsilon'\right)+ \delta,
\]
where the first inequality is due to the definition of our
oracle and the second one comes from the accuracy by which
our bisection procedure computes $z_{k+1}$.
Observe that the oracle might return a point
$\hat z_k$ such that $f(\hat z_k)$ is smaller than the
above right-hand side. In this case, no progress is done.
As $u_k \le f(\hat z_k)$, this implies:
\begin{equation}\label{eq:termination bound}
(\lambda+\lambda\alpha)l_k+(1+\alpha)\epsilon' + \delta\geq
(\lambda+\lambda\alpha -\alpha)f(\hat z_k).
\end{equation}
Using that $\hat f^*\geq l_k$ we get an upper bound of the
left-hand side.
Rearranging the terms and replacing $\lambda$ by its value, we get:
\[
\hat f^* + (2+\alpha)(\alpha \hat f^*+(1+\alpha)\epsilon' +
\delta)\geq f(\hat z_k).
\]
Since all the inequalities in the above derivation can be
tight, a better accuracy cannot be guaranteed with our
strategy.
Thus, we can output $\hat z_k$. \\
{\it Part (b).} Note that we can assume $\frac{f(\hat
z_0)-f(z_0)}{f(z_0)\epsilon}>1$, for otherwise the point $\hat z_0$ already satisfies our stopping criterion.

In order to assess the progress of the algorithm, we can
assume that the stopping criterion
\eqref{eq:termination bound} is not satisfied.
As $l_{k+1} = l_k$ in
our case where the oracle gives an output, we get:
\begin{eqnarray*}
\Delta_{k+1} &=& u_{k+1}-l_k  \le f(\hat z_{k+1})-l_k \\
&\leq& (1+\alpha)\left(\lambda l_k+(1-\lambda)u_k+\epsilon'\right)+ \delta-l_k\\
&=& \frac{\alpha^2+\alpha-1}{2+\alpha}l_k + \frac{1+\alpha}{2+\alpha}u_k + (1+\alpha)\epsilon'+
\delta\\
&=&\frac{1+\alpha}{2+\alpha}(u_k-l_k) + \alpha l_k +(1+\alpha)\epsilon'+
\delta\\
&\leq&\frac{1+\alpha}{2+\alpha}\Delta_k + \alpha \hat f^* +(1+\alpha)\epsilon'+
\delta.
\end{eqnarray*}

Suppose now that the oracle informs us that there is no mixed-integral
point with a value smaller than
$ f(z_{k+1})\geq\lambda l_k +(1-\lambda) u_k - \epsilon'$.
Then $\hat z_{k+1} = \hat z_k$ and $u_{k+1} = u_k$.
We have:
\begin{eqnarray*}
\Delta_{k+1} &=& u_{k+1}-l_{k+1} = f(\hat
z_k)-f(z_{k+1})\\
&\leq&
u_k-\left(\lambda l_k +(1-\lambda) u_k - \epsilon'\right)= \lambda\Delta_k + \epsilon'\\
&\leq& \frac{1+\alpha}{2+\alpha}\Delta_k + \alpha\hat f^* +(1+\alpha)\epsilon'+
\delta.
\end{eqnarray*}
The above inequality is valid for every $k$ that does not comply with the stopping criterion, whatever the oracle detects. 
Therefore, we get:
\[
\Delta_N\leq \left(\frac{1+\alpha}{2+\alpha}\right)^N\Delta_0 + (2+\alpha)
\left(\alpha\hat f^* + (1+\alpha)\epsilon'+ \delta \right),
\]
and the proposition is proved because $f(\hat z_N)-\hat f^*\leq
\Delta_N$. \qed \end{proof}

In the remainder of this paper, we elaborate on possible
realizations of our hard oracle.

We proceed as follows.
In Section \ref{sec.2D}, we focus on the special case when
$n=2$ and $d=0$.
We present a geometric construction that enables us to
implement the improvement oracle in polynomial time.
With the help of this oracle we then solve the problem
\eqref{eq:general_MIP_f} with $n=2$ and $d=0$  and obtain a
``best point'', i.e., an optimal point.
An adaptation of this construction can also be used to
determine a second and, more generally, a ``$k$-th best point''.
These results will be extended in Section~\ref{sec.2D+d} to
the mixed-integer case with two integer variables and $d$
continuous variables.
The latter extensions are then used as a subroutine to solve
the general problem \eqref{eq:general_MIP_f} with arbitrary $n$ and $d$ in finite time.

\section{Two-dimensional integer convex optimization} \label{sec.2D}

If $n=1$ and $d=0$, an improvement oracle can be trivially
realized for $\alpha = \delta = 0$.
Queried on a point $z\in\R$ the oracle returns $\hat z:=\arg\min\{f(\lfloor
z\rfloor),f(\lceil z\rceil)\}$ if one of these numbers is smaller or equal to $f(z)$, or returns \textbf{b} otherwise.
\comment{
Of course, minimizing a convex function over $\Z$ is no harder than
minimizing the same function over $\R$ (note that in general this
problem can be difficult because it is unbounded). Later in the paper, we will also use the obvious fact that the minimum of $f$ over $\Z\setminus\{\hat x^*\}$ is easy to determine: that point is one of the two neighbors of $\hat x^*$.

If $n=1$, the oracle can be easily realized if we have at our disposal a fast way of computing values of $g(\hat x):=\min\{f(x,y): (x,y)\in S \cap(\{\hat x\}\times\R^d)\}$.}
The first non-trivial case arises when $n=2$ and $d=0$.
This is the topic of this section.

\subsection{Minimizing a convex function in two integer variables}

We show in this section how to implement efficiently the
oracle $\mathrm{O}_{\alpha,\delta}$ with $\alpha=\delta=0$,
provided that the feasible set is contained in a known
finite box $[-B,B]^2$.

\begin{theorem}\label{2dim::theorem}
Let $f : \R^2 \to \R$ and $g_i : \R^2 \to \R$ with $i=1,\dots,m$ be convex functions.
Let $B\in\N$ and let $x\in[-B,B]^2$ such that $g_i(x)\le 0$ for all $i=1,\dots,m$. Then, in a number of evaluations of $f$ and $g_1,\dots,g_m$ that is polynomial in $\ln(B)$, one can either
\begin{enumerate}[(a)]
  \item find an $\hat x\in [-B,B]^2\cap\mathbb{Z}^2$ with $f(\hat x)\leq f(x)$ and $g_i(\hat x)\leq 0$ for all $i=1,\dots,m$ or
  \item show that there is no such point.
\end{enumerate}
\end{theorem}

Note that we do not allow for the function $f$ to take
infinite values, in order to ensure that we can minimize
$f$ over the integers of any segment of
$[-B,B]^2$ in $\OO(\ln(B))$ evaluations of $f$
using a bisection method.
Indeed, if a convex function takes infinite values, it can
cost up to $\OO(B)$ evaluations of $f$ to
minimize it on a segment containing 
$\OO(B)$ integer points, as there could be only
one of those points on its domain.

The algorithm that achieves the performance
  claimed in Theorem~\ref{2dim::theorem} is described in
the proof of the theorem.
That proof requires two lemmata. 
We use the following notation.
Let $Q\subset\R^2$.
We denote by $\vol(Q)$ the volume of $Q$, i.e., its Lebesgue
measure.
By $\aff\{Q\}$ we denote the smallest affine space
containing $Q$ and by $\conv\{Q\}$ the convex hull of $Q$.
The dimension $\dim(Q)$ of $Q$ is the dimension of
$\aff\{Q\}$.
The scalar product we use in this section is exclusively the
standard dot product.


\begin{lemma}\label{2dim::lowdimension}
Let $K \subset \R^2$ be a polytope with $\vol(K) < \frac{1}{2}$.
Then $\dim(\conv(K \cap \Z^2))\le1$.
\end{lemma}
\begin{proof}
\smartqed
For the purpose of deriving a contradiction, assume that there exist
three affinely independent points $\hat x,\hat y,\hat z\in K\cap\Z^2$.
Then $\vol(K)\ge\vol(\conv(\{\hat x,\hat y,\hat z\})) = \frac{1}{2}|\det(\hat x-\hat z,\hat y-\hat z)|\ge\frac{1}{2}$.
\qed\end{proof}

\begin{lemma}\label{2dim::shrinking}
Let $u,v,w\in\R^2$ be affinely independent. If
\[\big(\conv\{u,u+v,u+v+w\}\setminus\left(\conv\{u+v,u+v+w\}\cup\{u\}\right)\big)\cap\Z^2=\emptyset,\]
then the lattice points $\conv\{u,u+v,u+v-w\}\cap\Z^2$ lie on at
most three lines.
\end{lemma}

\begin{proof}
\smartqed
We partition $\conv\{u,u+v,u+v-w\}$ into three regions.
Then we show that in each region the integer points must lie on
a single line using a lattice covering argument.

We define the parallelogram
$P:=\conv\{0,\frac{1}{2}v,\frac{1}{2}w,\frac{1}{2}v+\frac{1}{2}w\}$.
Further, we set \[A_1:=u-\frac{1}{2}w+P,\quad A_2:=u+\frac{1}{2}v-w+P, \quad\textrm{and}\quad
A_3:=u+\frac{1}{2}v-\frac{1}{2}w+P.\] Note that
$\conv\{u,u+v,u+v-w\}\subset A_1\cup A_2\cup A_3$ (see Fig.~\ref{fig:figure1}). Our assumption implies that the set $u+\frac{1}{2}v+P$ does not contain any integer point except possibly on the segment $u+v+\conv\{0,w\}$. Therefore, for a sufficiently small $\varepsilon>0$, the set $(u+\frac{1}{2}v-\varepsilon(v+w)+P)\cap \Z^2$ is empty. 

Assume now that one of the three regions, say $A_1$, contains three affinely
independent integer points $\hat x,\hat y,\hat z$. We show below that $A_1+\Z^2=\R^2$,
i.e., that $P$ defines a lattice covering, or equivalently that the set $t+P$ contains at least one integer point for every $t\in\R^2$.
This fact will contradict that
$(u+\frac{1}{2}v-\varepsilon(v+w)+P)\cap \Z^2=\emptyset$ and thereby prove the lemma.

\begin{figure}[ht]
\centering
\begin{tikzpicture}[scale=0.8]
\tikzstyle{every node}=[circle, fill=black, draw,inner sep=0pt,
minimum width=2pt]
 \node (a)  [label=below:$u$,draw]{} 
node (b) at (-0.5,4) [label=left:$u+v-w\,$]{} 
node (c) at (2.5,4) [label=above:$u+v$]{} 
node (d) at (5.5,4) [label=right:$\,u+v+w$]{} ;
\draw {(a) -- (b) -- (c) -- (d) -- (a) -- (c)}; 
\draw [dashed]{(a) -- (-1.5,0) (-1.5,0) -- (1,4) (1.25,2) -- (-1.75,2) (-1.75,2) -- (b)}; 
\draw (-1,0.5 ) node[draw=none,fill=none,label=right:$A_1$]{}; 
\draw (0.25,2.5 ) node[draw=none,fill=none,label=right:$A_3$]{}; 
\draw (-1.25,2.5 ) node[draw=none,fill=none,label=right:$A_2$]{};
\end{tikzpicture}
\caption{Partitioning the triangle in regions.} \label{fig:figure1}
\end{figure}
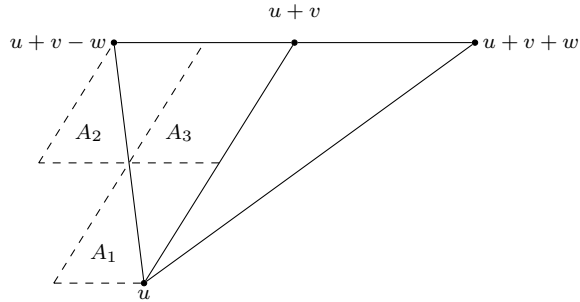
\begin{figure}[ht]
\begin{minipage}[b]{0.50\linewidth}
\centering
\begin{tikzpicture}[scale=0.8]
\tikzstyle{every node}=[circle, fill=black, draw,inner sep=0pt,
minimum width=2pt] \fill[color=gray!75] (4.75,3.5) -- (6.0,2.75) --
(4.28125,2.75) -- cycle; \fill[color=gray!50] (6.0,2.75) --
(4.28125,2.75) -- (3.126865672,0.902985075) -- cycle; \fill[color=gray!25]
(2.5,0.5) -- (1.25,1.25) -- (4.75,3.5) -- (3.126865672,0.902985075) -- cycle;
\node (x) at (2.5,0.5) [label=below:$\hat x$,draw]{} node (z) at
(4.75,3.5) [label=above:$\hat z$]{} node (y) at (1.25,1.25)
[label=left:$\hat y$]{} node (d) at (6.0,2.75) [label=right:$\,\hat x-\hat y+\hat z$]{}
node (j) at (3.126865672,0.902985075) {};
\draw {(x) -- (y) -- (z) -- (d) -- (x)}; 
\draw {(z) -- (j)};
\draw {(d) -- (4.28125,2.75)}; 
\draw [dashed]{(-0.3,-0.16) -- (2.9,-0.16)  -- (5.6,4.16) -- (2.3,4.16) -- (-0.3,-0.16)}; 
\draw
(3.2,0.3 ) node[draw=none,fill=none,label=right:$A_1$]{}; \draw
(3.5,4.66 ) node[draw=none,fill=none,label=right:$w^\perp$]{}; \draw
(1.0,3) node[draw=none,fill=none,label=below:$v^\perp$]{};
\draw[arrows=->,line width=.4pt](3.5,4.16)--(3.5,5.16);
\draw[arrows=->,line width=.4pt](1.05,2)--(0.202001695994912,2.52999894000318);
\end{tikzpicture}
\end{minipage}
\hspace{0.5cm}
\begin{minipage}[b]{0.50\linewidth}
\centering
\begin{tikzpicture}[scale=0.8]
\tikzstyle{every node}=[circle, fill=black, draw,inner sep=0pt,
minimum width=2pt] \fill[color=gray!50] (4.75,3.5) -- (3.03125,3.5) --
(1.876865672,1.652985075) -- cycle; \fill[color=gray!75] (2.5,0.5)
-- (1.25,1.25) -- (0.78125,0.5) -- cycle; \fill[color=gray!25]
(2.5,0.5) -- (1.25,1.25) -- (4.75,3.5) -- (3.126865672,0.902985075) -- cycle;
\node (x) at (2.5,0.5) [label=below:$\hat x$,draw]{} node (z) at
(4.75,3.5) [label=above:$\hat z$]{} node (y) at (1.25,1.25)
[label=left:$\hat y$]{} node (j) at (3.126865672,0.902985075) {};
\draw {(x) -- (y) -- (z) -- (j) -- (x)}; 
\draw {(x) -- (0.78125,0.5) -- (y)};
\draw {(z) -- (3.03125,3.5) -- (1.876865672,1.652985075)};
\draw [dashed]{(-0.3,-0.16) -- (2.9,-0.16)  -- (5.6,4.16) -- (2.3,4.16) -- (-0.3,-0.16)}; 
\draw
(3.2,0.3 ) node[draw=none,fill=none,label=right:$A_1$]{}; \draw
(3.5,4.66 ) node[draw=none,fill=none,label=right:$w^\perp$]{}; \draw
(1.0,3) node[draw=none,fill=none,label=below:$v^\perp$]{};
\draw[arrows=->,line width=.4pt](3.5,4.16)--(3.5,5.16);
\draw[arrows=->,line width=.4pt](1.05,2)--(0.202001695994912,2.52999894000318);
\end{tikzpicture}
\end{minipage}
\caption{Mapping $T$.}\label{fig:figure2}
\end{figure}
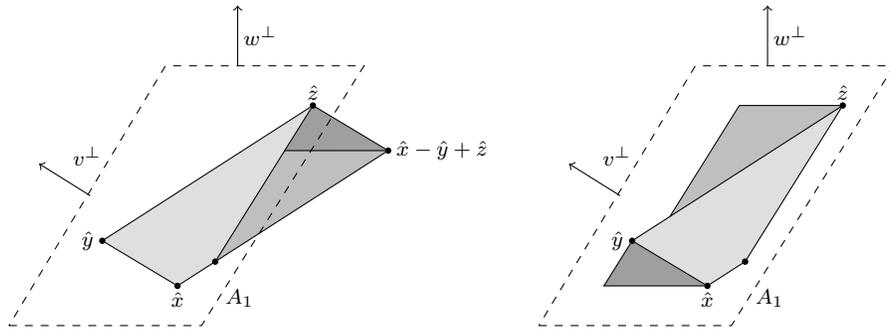

Clearly, the parallelogram $Q:=\conv\{ \hat x,\hat y,\hat z,\hat x-\hat y+\hat z \}$ defines a
lattice covering, as it is full-dimensional and its vertices are integral. 
We transform $Q$ into a set $Q'\subseteq A_1$ for
which $a\in Q'$ iff there exists $b\in Q$ such that $a-b\in\Z^2$.
Specifically, we define a mapping $T$ such that $Q'=T(Q)\subset A_1$ and
$T(Q)+\Z^2=\R^2$. Let $v^\perp:=(-v_2,v_1)^\top$ and
$w^\perp:=(-w_2,w_1)^\top$, i.e., vectors orthogonal to $v$ and $w$.
Without loss of generality (up to a permutation of the names $\hat x,\hat y,\hat z$), 
we can assume that $\langle\hat x, w^\perp\rangle\le \langle\hat y, w^\perp\rangle \le \langle\hat z,
w^\perp\rangle$. If $\hat x-\hat y+\hat z\in A_1$ there is nothing to show, so we suppose that
$\hat x-\hat y+\hat z\notin A_1$. 

Note that $\langle\hat x, w^\perp\rangle\le \langle\hat x-\hat y+\hat z, w^\perp
\rangle\le \langle\hat z, w^\perp\rangle$. Assume first
that $\langle\hat x-\hat y+\hat z, v^\perp\rangle<\langle\hat z,v^\perp\rangle\le \langle\hat x,v^\perp\rangle,\langle\hat y, v^\perp\rangle$ --- the strict inequality resulting from the fact that $\hat x-\hat y+\hat z\notin A_1$. We define the mapping $T:Q\to A_1$ as
follows,
\begin{equation*}
T(l)=
\begin{cases}
  l+\hat y-\hat z,    & \text{if } \langle l,v^\perp\rangle < \langle \hat z,v^\perp\rangle \text{ and } \langle l,w^\perp\rangle > \langle \hat x-\hat y+\hat z,w^\perp\rangle, \\
  l-\hat x+\hat y,    & \text{if } \langle l,v^\perp\rangle < \langle \hat z,v^\perp\rangle \text{ and } \langle l,w^\perp\rangle \le \langle \hat x-\hat y+\hat z,w^\perp\rangle, \\
  l,        & \text{otherwise}
\end{cases}
\end{equation*}
(see Fig.~\ref{fig:figure2}). It is straightforward to show that $T(Q)\subset A_1$ and $T(Q)+\Z^2=\R^2$.
A similar construction can easily be defined for any
possible ordering of $\langle \hat x-\hat y+\hat
z,v^\perp\rangle$, $\langle \hat z,v^\perp\rangle$,
$\langle\hat x,v^\perp\rangle$, and $\langle \hat
y,v^\perp\rangle$.\qed \end{proof}

\begin{remark}
In each region $A_i$, the line containing $A_i\cap\Z^2$, if
it exists, can be computed by the minimization of an
arbitrary linear function $x\mapsto \langle c,x\rangle$ over
$A_i\cap\Z^2$, with $c\neq 0$, and the maximization  of the
same function with the fast algorithm described in
\cite{Heisenbrand2d}.
If these problems are feasible and yield two distinct
solutions, the line we are looking for is the one joining
these two solutions.
If the two solutions coincide, that line is the one
orthogonal to $c$ passing through that point.

The algorithm in \cite{Heisenbrand2d} is applicable to
integer linear programs with two variables and $m$ constraints. The data of the problem should be integral. This algorithm runs in $\OO(m+\phi)$, where $\phi$ is the binary encoding length of the data.$\hfill\diamond$
\end{remark}

\begin{proof}[of Theorem \ref{2dim::theorem}]
\smartqed
As described at the beginning of this subsection, a one-dimensional integer minimization problem
can be solved polynomially with respect to the logarithm of the
length of the segment that the function is optimized over.
In the following we explain how to reduce the implementation of the two-dimensional oracle to the task of solving one-dimensional integer minimization problems.
For notational convenience, we define $g(y):=\max_{i=1\dots m}g_i(y)$ for $y\in\R^2$ which is again a convex function.

Let $F_1,\ldots,F_4$ be the facets of $[-B,B]^2$.
Then $[-B,B]^2=\bigcup_{j=1}^{4}\conv\{x,F_j\}$.
The procedure we are about to describe has to be applied to
every facet $F_1,\ldots,F_4$ successively, until a suitable
point $\hat x$ is found.
Let us only consider one facet $F$.
We define the triangle $T_0:=\conv\{x,F\}$, whose
area is smaller than $2B^2$.

To find an improving point within $T_0$, we construct a sequence
$T_0\supset T_1\supset T_2 \supset \dots$ of triangles that all have $x$ as
vertex, with $\vol(T_{k+1})\le \frac{2}{3}\vol(T_k)$, and such that $f(\hat y)>f(x)$ or $g(\hat y)>0$ for all $\hat y\in(T_0\setminus
T_k)\cap\Z^2$. We stop our search if we have found an $\hat x\in[-B,B]^2\cap\Z^2$ such that
$f(\hat x)\le f(x)$ and $g(\hat x)\le 0$, or if the volume of one of the triangles $T_k$ is smaller than
$\frac{1}{2}$.
The latter happens after at most
$k=\lceil\ln(4B^2)/\ln(\frac{3}{2})\rceil$ steps.
Then, Lemma~\ref{2dim::lowdimension} ensures that the
integral points of $T_k$ are on a line, and we
need at most $\OO(\ln(B))$ iterations to solve the resulting
one-dimensional problem.

The iterative construction is as follows. Let $T_k=\conv\{x,v_0,v_1\}$ be given. We write
$v_\lambda:=(1-\lambda)v_0+\lambda v_1$ for $\lambda\in\R$ and we
define the auxiliary triangle
$\bar{T}_k:=\conv\{x,v_{1/3},v_{2/3}\}$.
Consider the integer linear program
\begin{equation}\label{2dim::IP}
\min\{\langle h,\hat y\rangle: \hat y \in \bar{T}_k\cap\Z^2\}
\end{equation}
where $h$ is the normal vector to $\conv\{v_0,v_1\}$ such that $\langle h,x\rangle <
\langle h,y\rangle$ for every $y\in F$. We distinguish two cases.

{\bf Case 1.} The integer linear program \eqref{2dim::IP} is
infeasible. Then $\bar{T}_k\cap\Z^2=\emptyset$. It remains to check
for an improving point within $(T_k\setminus\bar{T}_k)\cap\Z^2$. By
construction, we can apply Lemma \ref{2dim::shrinking} twice (with
$(u,u+v-w,u+v+w)$ equal to $(x,v_0,v_{2/3})$ and $(x,v_{1/3},v_1)$, respectively) to
determine whether there exists an $\hat x \in(T_k\setminus\bar{T_k}) \cap
\Z^2$ such that $f(\hat x)\le f(x)$ and $g(\hat x)\le 0$. This requires to solve at most six
one-dimensional subproblems.

{\bf Case 2.} The integer linear program (\ref{2dim::IP})
has an optimal solution $\hat z$. If $f(\hat z)\le f(x)$ and $g(\hat z)\le 0$, we are done. So we assume that $f(\hat z) > f(x)$ or $g(\hat z)> 0$.
Define $H:=\{y\in\R^2\;|\;\langle h,y\rangle=\langle h,\hat z\rangle\}$,
that is, the line containing $\hat z$ that is parallel to $\conv\{v_0,v_1\}$, 
and denote by $H_+$ the closed half-space with boundary $H$
that contains $x$.
By definition of $\hat z$, there is no integer
point in $\bar{T}_k\cap \inter H_+$.
Further, let $L:=\aff\{x,\hat z\}$.

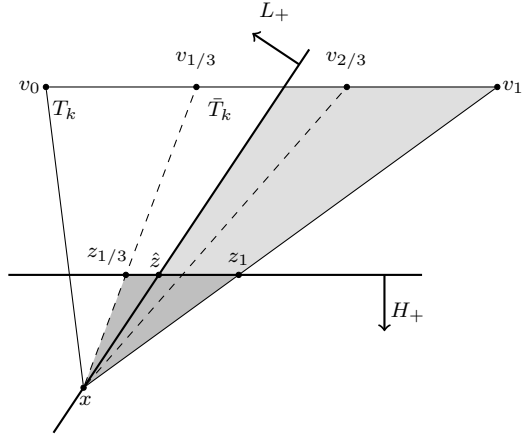
\begin{figure}[ht]
\centering
\begin{tikzpicture}[scale=1]
\tikzstyle{every node}=[circle, fill=black, draw,inner sep=0pt,
minimum width=2pt] \fill[color=gray!50] (0.5625,1.5) -- (2.0625,1.5)
-- (0,0) -- cycle; \fill[color=gray!25] (1,1.5) -- (2.0625,1.5) --
(5.5,4) -- (2.6666666666,4) -- cycle; \node (y) [label=below:$\,x$,draw]{} node
(v0) at (-0.5,4) [label=left:$v_0$]{} node (v1) at (1.5,4)
[label=above:$v_{1/3}$]{} node (v2) at (3.5,4)
[label=above:$v_{2/3}$]{}  node (v3) at (5.5,4) [label=right:$v_1$]{}
node (a) at (0.5625,1.5) [label=above left:$z_{1/3}$] {} node (b) at
(2.0625,1.5) [label=above:$z_1$] {} node (z) at (1,1.5)
[label=above:$\hat z\,\,{}$]{}; \draw {(y) -- (v0) -- (v3) -- (y)}; \draw
(-0.5,3.7) node[draw=none,fill=none,label=right:$T_k$]{}; \draw
[dashed]{(y) -- (v1)}; \draw [dashed]{(y) -- (v2)};\draw (1.55,3.7)
node[draw=none,fill=none,label=right:$\bar T_k$]{};
\draw[thick]{(-1,1.5) -- (4.5,1.5)}; \draw[thick]{(-0.4,-0.6) --
(3,4.5)};\draw[->,thick]{(4,1.5) -- (4,0.75)};
\draw[->,thick]{(2.866667,4.3) --
(2.242628945913283,4.716025147168922)}; \draw (2.25,5)
node[draw=none,fill=none,label=right:$L_+$]{}; \draw (4,1)
node[draw=none,fill=none,label=right:$H_+$]{};
\end{tikzpicture}
\caption{Illustration of Case 2.} \label{fig:figure3}
\end{figure}

Due to the
convexity of the set $\{y\in\R^2 \;|\; f(y)\le f(x),\; g(y)\le 0\}$ and the fact that $f(\hat z)>f(x)$ or $g(\hat z)> 0$, there
exists a half-space $L_+$ with boundary $L$ such that the possibly empty segment $\{y\in H \;|\;
f(y)\le f(x),\;g(y)\le 0\}$ lies in $L_+$ (see Fig.~\ref{fig:figure3}). By convexity of $f$ and $g$, the set $((T_k\setminus
H_+) \setminus L_+)$ (the lightgray region in Fig.~\ref{fig:figure3}) contains no point $y$ for which
$f(y) \leq f(x)$ and $g(y)\le 0$.
It remains to check for an improving point within
$((T_k\cap H_+) \setminus L_+)\cap\Z^2$.
For that we apply again Lemma~\ref{2dim::shrinking} on the
triangle $\conv\{z_{1/3},z_1,x\}$ (the darkgray region
in Fig.~\ref{fig:figure3}), with $z_{1/3} =
H\cap\aff\{x,v_{1/3}\}$ and $z_1 = H\cap\aff\{x,v_1\}$.
If none of the corresponding subproblems returns a suitable
point $\hat x \in \Z^2$, we know that $T_k\setminus L_+$ contains no
improving integer point. Defining $T_{k+1}:=T_k\cap L_+$, we have by construction
$f(\hat y)> f(x)$ or $g(\hat y)>0$ for all
$\hat y\in(T_k\setminus T_{k+1})\cap\Z^2$ and 
$\vol(T_{k+1})\le \frac{2}{3}\vol(T_k)$. 

It remains to determine the half-space $L_+$.
If $g(\hat z) > 0$ we just need to find a point $y\in H$
such that $g(y)<g(\hat z)$, or if $f(\hat z)>f(x)$, it
suffices to find a point $y\in H$ such that $f(y)<f(\hat
z)$.
Finally, if we cannot find such a point $y$ in either case,
convexity implies that there is no suitable point in
$T_k\setminus H_+$; another application of
Lemma~\ref{2dim::shrinking} then suffices to determine
whether there is a suitable $\hat x$ in $T_k\cap
H_+\cap\Z^2$.
\qed\end{proof}

The algorithm presented in the proof of Theorem \ref{2dim::theorem}  can be adapted to output a minimizer $\hat x^*$ of $f$ over $S\cap[-B,B]^2\cap\Z^2$, provided that we know in advance that the input point $x$ satisfies $f(x)\leq \hat f^*$: it suffices to store and update the best value of $f$ on integer points found so far. 
In this case the termination procedure is not necessary.

\begin{corollary}\label{2dim::optimization}
Let $f : \R^2 \to \R$ and $g_i : \R^2 \to \R$ with $i=1,\dots,m$ be convex functions.
Let $B\in\N$ and let $x\in[-B,B]^2$ such that $g_i(x)\le 0$ for all $i=1,\dots,m$. If $f(x)\le \hat f^*$, then, in a number of evaluations of $f$ and $g_1,\dots,g_m$ that is polynomial in $\ln(B)$, one can either
\begin{enumerate}[(a)]
  \item find an $\hat x\in [-B,B]^2\cap\mathbb{Z}^2$ with $f(\hat x)= \hat f^*$ and $g_i(\hat x)\leq 0$ for all $i=1,\dots,m$ or
  \item show that there is no such point.
\end{enumerate}
\end{corollary}

\LinesNumbered
\begin{algorithm}[htbp]
\DontPrintSemicolon 
\KwData{$x\in[-B,B]^2$ with $f(x)\le\hat f^*$ and $g_i(x)\leq 0$ for all $i=1,\dots,m$.} 
Let $F_1,\ldots,F_4$ be the facets of $[-B,B]^2$.\;
Set $\hat x^*:=0$ and $\hat f^*:=+\infty$.\;
\For{$t=1,\dots,4$}{
	Set $F:=F_t$ and define $v_0,v_1\in\R^n$ such that $F:=\conv\{v_0,v_1\}$.\;
	Write $h$ for the vector normal to $F$ pointing outwards $[-B,B]^2$.\;
	Set $T_0:=\conv\{x,F\}$ and $k:=0$.\;
	\While{$\vol(T_k)\geq \frac{1}{2}$}{
		Set $\bar T_k:=\conv\{x,v_{1/3},v_{2/3}\}$, with $v_\lambda:=(1-\lambda)v_0 + \lambda v_1$.\;
		Solve $(\mathcal{P}):\min\{\langle h,\hat y\rangle: \hat y\in\bar T_k\cap\Z^2\}$.\label{algo:2d:ip}\label{algo:2d:lp}\;
		\uIf{(Case 1) $(\mathcal{P})$ is infeasible,}{
			Determine $\hat x := \arg\min\{f(\hat z)\;|\;\hat z\in(T_k\setminus\bar T_k)\cap\Z^2\text{ with }g(\hat z)\le0\}$.\label{algo:2d:shrinking:1}\;
			\lIf{$\hat x$ exists and $f(\hat x)<\hat f^*$}{Set $\hat x^*:=\hat x$ and $\hat f^*:=f(\hat x)$.\;}
			}
		\Else{
			{\it (Case 2)} Let $\hat z$ be an optimal solution of $(\mathcal{P})$.\;
			Set $H_+:=\{y\in\R^2:\langle h,y\rangle\leq\langle h,\hat z\rangle\}$ and $H:=\partial H_+$.\;
			Define the points $v := \aff\{x,\hat z\}\cap F$ and $z_i=H\cap\conv\{x,v_i\}$ for $i = 0,1$.\;
			Denote $z_\lambda := (1-\lambda)z_0 + \lambda z_1$ for $\lambda\in(0,1)$.\;
				\uIf{$g(\hat z)\le0$ {\bf and} there is a $y\in\conv\{z_0,\hat z\}$ for which $f(y)<f(\hat z)$ {\bf or}\label{algo:2d:left1}\;
				       $\;\;\;\;g(\hat z)>0$ {\bf and} there is a $y\in\conv\{z_0,\hat z\}$ for which $g(y)<g(\hat z)$ \label{algo:2d:left2}}{
					Determine $\hat x := \arg\min\{f(\hat z)\;|\;\hat z\in\conv\{x,z_{1/3},z_1\}\cap\Z^2\text{ with }g(\hat z)\le0\}$.\label{algo:2d:shrinking:2}\;
					\lIf{$\hat x$ exists and $f(\hat x)<\hat f^*$}{Set $\hat x^*:=\hat x$ and $\hat f^*:=f(\hat x)$.\;}
					Set $v_1:=v$, $T_{k+1}:=\conv\{x,v_0,v\}$, and $k:=k+1$.\;
					}
				\Else{
					Determine $\hat x := \arg\min\{f(\hat z)\;|\;\hat z\in\conv\{x,z_0,z_{2/3}\}\cap\Z^2\text{ with }g(\hat z)\le0\}$.\label{algo:2d:shrinking:3}\;
					\lIf{$\hat x$ exists and $f(\hat x)<\hat f^*$}{Set $\hat x^*:=\hat x$ and $\hat f^*:=f(\hat x)$.\;}
					Set $v_0:=v$, $T_{k+1}:=\conv\{x,v,v_1\}$, and $k:=k+1$.\;
					}

			}	
		}
		Determine $\hat x := \arg\min\{f(\hat z)\;|\;\hat z\in T_k\cap\Z^2\text{ with }g(\hat z)\le0\}$.\label{algo:2d:lowdimension}\;
		\lIf{$\hat x$ exists and $f(\hat x)<\hat f^*$}{Set $\hat x^*:=\hat x$ and $\hat f^*:=f(\hat x)$.\;}
	}
\lIf{$\hat f^*<+\infty$}{Return {$\hat x^*$}.\;}
\lElse{Return ``the problem is unfeasible''.} \vspace*{2mm}
\caption{Minimization algorithm for 2D problems.\label{algo:2d}} 
\end{algorithm}

Note that line \ref{algo:2d:lowdimension} in Algorithm~\ref{algo:2d} requires the application of Lemma \ref{2dim::lowdimension}.
Lines \ref{algo:2d:shrinking:1}, \ref{algo:2d:shrinking:2} and \ref{algo:2d:shrinking:3} require the application of Lemma \ref{2dim::shrinking}. 

\begin{remark}[Complexity]\label{rem:complexity2d}
The following subroutines are used in Algorithm~\ref{algo:2d}.
\begin{description}
 \item[Line \ref{algo:2d:lp} and applications of Lemma
   \ref{2dim::shrinking}. ] A two-dimensional integer linear
   program solver for problems having at most four constraints, such as the one described in \cite{Heisenbrand2d}. The size of the data describing each of these constraints is in the order of the representation of the vector $x$ as a rational number, which, in its standard truncated decimal representation, is in $\OO(\ln(B))$.
 \item[Line \ref{algo:2d:lowdimension} and applications of Lemma
   \ref{2dim::shrinking}. ] A solver for one-dimensional
   integer convex optimization problems. At every iteration, we need to perform at most seven of them, for a cost of $\OO(\ln(B))$ at each time.
 \item[Lines \ref{algo:2d:left1} and \ref{algo:2d:left2}.]
   Given a segment $[a,b]$ and one of its points $z$, we
   need a device to determine which of the two regions
   $[a,z]$ or $[z,b]$ intersects a level set
   defined by $f$ and $g$ that does not contain $z$.
   This procedure has a complexity of $\OO(\ln(B))$ and only
   occurs in Case 2 above. $\hfill\diamond$
\end{description}
\end{remark}

\subsection{Finding the $k$-th best point}\label{subsection:k-th point}

In this subsection we want to show how to find the $k$-th best point, provided that the $k-1$ best points are known.
A slight variant of this problem will be used in Subsection~\ref{sec:MIPgen} as a subroutine for the general mixed-integer convex problem.
In the following, we describe the necessary extensions of the previous Algorithm~\ref{algo:2d}.
Let $\hat x_1^*:=\hat x^*$ and define for $k\geq 2$:
\begin{equation*}
\hat x_k^*:=\arg\min\left\lbrace f(\hat x) \;|\; \hat x\in (S\cap [-B,B]^2\cap \mathbb{Z}^2)\setminus \{\hat x_1^*,\dots,\hat x_{k-1}^*\}\right\rbrace
\end{equation*}
to be the $k$-th best point.
Observe that, due to the convexity of $f$ and $g_1,\dots,g_m$, we can always assume that $\conv\{\hat x_1^*,\dots,\hat x_{k-1}^*\}\cap\mathbb{Z}^2=\{\hat x_1^*,\dots,\hat x_{k-1}^*\}$ for all $k\geq 2$.
Although this observation appears plausible it is not completely trivial to achieve this algorithmically.

\begin{lemma}\label{lem:P_k cap Z^2 = S_k}
Let $\Pi_j:=\{\hat x^*_1,\ldots,\hat x^*_j\}$ be the ordered $j$ best points of our problem and $P_j$ be the convex hull of $\Pi_j$. Suppose that, for a given $k\geq 2$, we have $P_{k-1}\cap\Z^2 = \Pi_{k-1}$. Let $\hat x^*_k$ be a $k$-th best point.
\begin{enumerate}[(a)]
 \item If $f(\hat x_k^*)>\hat f^*$, we can replace the point $\hat x_k^*$ by a feasible $k$-th best point $\hat z_k^*$ such that $\conv\{\Pi_{k-1},\hat z^*_k\} \cap\Z^2 = \{\Pi_{k-1},\hat z^*_k\}$
 in $\OO(1)$ operations.
 \item If $f(\hat x_k^*)=\hat f^*$, and if we have at our disposal the $\nu$ vertices of $P_{k-1}$ ordered counterclockwise, we can construct such a point $\hat z_k^*$ in $\OO(\nu\ln(B))$ operations.
\end{enumerate}
\end{lemma}
\proof
\smartqed
{\it Part (a).} Suppose first that $f(\hat x^*_k)>\hat f^*$, and assume that we cannot set $\hat z^*_k:=\hat x^*_k$, that is, that there exists $\hat x\in (P_k\cap\Z^2)\setminus \Pi_k$. Then $\hat x = \sum_{i=1}^k\lambda_i\hat x_i^*$ for some $\lambda_i\geq 0$ that sum up to $1$. Note that $0<\lambda_k<1$, because $\hat x\notin P_{k-1}\cup\{\hat x^*_k\}$ by assumption, and that $f(\hat x)\geq f(\hat x^*_k)$. We deduce:
\[
0\leq f(\hat x)-f(\hat x^*_k) \leq \sum_{i=1}^k\lambda_i(f(\hat x_i^*)-f(\hat x^*_k))\leq 0.
\]
Thus $f(\hat x)=f(\hat x^*_k)$. Let $I:=\{i:\lambda_i>0\}$
and $Q_I:=\conv\{\hat x^*_i:i\in I\}$, so that $\hat
x\in\relint Q_I$.
Observe that $|I|\geq 2$ and that $f$ is constant on $Q_I$.
Necessarily, $Q_I$ is a segment.
Indeed, if it were a two-dimensional set, we could consider
the restriction of $f$ on the line $\ell := \aff\{\hat
x^*_1,\hat x\}$: it is constant on the open interval
$\ell\cap \inter Q_I$, but does not attain its minimum on
it, contradicting the convexity of $f$.
Let us now construct the point $\hat z^*_k$:
it suffices to consider the closest point to $\hat x^*_k$ in
$\aff\{Q_I\}\cap P_{k-1}$, say $\hat x_j^*$, and to take the
integer point $\hat z^*_k\neq\hat x_j^*$ of $\conv\{\hat
x_j^*,\hat x_k^*\}$ that is the closest to $\hat x_j^*$
(see Fig.~\ref{fig:k-th.best.point}).
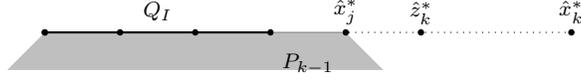
\begin{figure}[ht]
\centering
\begin{tikzpicture}[scale=1]
\tikzstyle{every node}=[circle, fill=black, draw,inner sep=0pt,
minimum width=2pt] 
\fill[color=gray!50] (0,0) -- (4,0) --
(4.5,-0.5) -- (-0.5,-0.5) -- cycle;
\node (x1) {}
node (xj) at (4,0) [label=above:$\hat x^*_j$,draw]{}
node (xk) at (7,0) [label=above:$\hat x^*_k$,draw]{} 
node (zk) at (5,0) [label=above:$\hat z^*_k$]{} 
node (x0) at (0,0) {}
node (x1) at (1,0) {}
node (x2) at (2,0) {}
node (x3) at (3,0) {}; 
\draw[thick] {(x0) -- (x3)}; \draw[gray] {(x3) -- (xj)}; \draw[dotted] {(xj) -- (xk)};
\draw(1.5,0) node[draw=none,fill=none,label=above:$Q_I$]{};
\draw(3.5,0) node[draw=none,fill=none,label=below:$P_{k-1}$]{};
\end{tikzpicture}
\caption{Illustration of Part (a).} \label{fig:k-th.best.point}
\end{figure}

\noindent {\it Part (b).}
Suppose now that $f(\hat x^*_i) = \hat f^*$ for every $1\leq i\leq k$, and define 
\[
\{\hat y_0^*\equiv \hat y^*_\nu,\hat y_1^*,\ldots,\hat y^*_{\nu-1}\}\subseteq \Pi_{k-1}
\] as the vertices of $P_{k-1}$, labeled
counterclockwise.
It is well-known that determining the convex hull of
$P_{k-1}\cup\{\hat x_k^*\}$ costs $\OO(\ln(\nu))$
operations.
From these vertices we deduce the set $\{\hat y^*_i:i\in J\}$ of those points that
are in the relative interior of that convex hull. Up to a
renumbering of the $\hat y^*_l$'s, we have $J =
\{1,2,\ldots,j-1\}$. We show below that Algorithm
\ref{algo:P_k cap Z^2} constructs a satisfactory point
$\hat z^*_k$.


Let us denote by $\hat x^*_k(i)$ the point $\hat x_k^*$ that the algorithm has at the beginning of iteration $i$ and define 
$T_l(i):=\conv\{\hat x^*_k(i),\hat y^*_l,\hat
y^*_{l+1}\} \setminus P_{k-1}$ for $0\leq l < j$ (see
Fig.~\ref{fig:figure_P_k_cap_Z^2}).
At iteration $i$, the algorithm considers the triangle $T_i(i)$ if its signed area $\frac{1}{2}\det(\hat x^*_k(i)-\hat y^*_i,\hat y^*_{i+1}-\hat y^*_i)$ is nonnegative, and finds a point $\hat x^*_k(i+1)\in T_i(i)$ such that $T_i(i+1)$ has only $\hat x^*_k(i+1)$ as integer point. 

We prove by recursion on $i$ that $T_l(i)$ contains only $\hat x^*_k(i)$ as integer point whenever $l<i$. We already noted it when $i = 0$. Suppose the statement is true for $i$, and let $l\leq i$. We have:
\[
\hat x^*_k(i+1)\in T_i(i)\subseteq\conv\{\hat x^*_k(i),\hat y^*_0,\ldots,\hat y^*_{i+1}\}\setminus P_{k-1}=T_i(i)\cup \bigcup_{l=0}^{i-1}T_l(i),
\]
hence
\[
K:=\conv\{\hat x^*_k(i+1),\hat y^*_0,\ldots,\hat y^*_{i+1}\}\setminus P_{k-1} \subseteq\conv\{\hat x^*_k(i),\hat y^*_0,\ldots,\hat y^*_{i+1}\}\setminus P_{k-1}.
\]
As $T_l(i+1)\subseteq K$ for all $l \le i$, the integers of $T_l(i+1)$ are either in $\bigcup_{l=0}^{i-1}T_l(i)\cap\Z^2$, which reduces to $\{\hat x^*_k(i)\}$ by recursion hypothesis, or in $T_i(i)$. Since $\hat x^*_k(i)\in T_i(i)$, all the integers in $T_l(i+1)$ must be in $T_i(i)$. But $T_l(i+1) \cap T_i(i)\cap\Z^2 = \{\hat x^*_k(i+1)\}$ by construction of $\hat x^*_k(i+1)$, and the recursion is proved.

It remains to take the largest value that $i$ attains in the
course of Algorithm~\ref{algo:P_k cap Z^2} to finish the
proof.
We need to solve at most $\nu-1$
two-dimensional integer linear
problems over triangles to compute $\hat x^*_k$; the data of
these problems are integers bounded by $B$.
\qed

\begin{algorithm}[htbp]
\LinesNumberedHidden
\DontPrintSemicolon
\KwData{$\hat x^*_k, \hat y_0^*,\hat y_1^*,\ldots,\hat y^*_j$.}
Set $i:=0$ and $x^*_k(0) := x^*_k$.\;
	\While{ $\det(\hat x^*_k(i)-\hat y^*_i,\hat y^*_{i+1}-\hat y^*_i)\geq 0$}{
	Set $T_i:=\conv\{\hat x^*_k(i),\hat y^*_i,\hat y^*_{i+1}\}\setminus\aff\{\hat y^*_i,\hat y^*_{i+1}\}$.\;
	Set $h_i$ a vector orthogonal to $\aff\{\hat y^*_i,\hat y^*_{i+1}\}$ such that $\langle h_i,\hat x^*_k(i)-\hat y^*_i\rangle>0$.\;
	Set $\hat x^*_k(i+1):=\arg\min\{\langle h_i,\hat
        x\rangle:\hat x \in T_i \cap \Z^2\}$.\;
	Set $i:=i+1$.\;
	}
        Set $\hat z^*_k := \hat x^*_k(i)$. \\ \quad
\caption{A point $\hat z^*_k$ for which $\conv\{\Pi_{k-1},\hat z_k^*\}\cap\mathbb{Z}^2=\{\Pi_{k-1},\hat z_k^*\}$.\label{algo:P_k cap Z^2}} 
\end{algorithm}

\begin{figure}[ht]
\centering
\begin{tikzpicture}[scale=0.5]
\tikzstyle{every node}=[circle, fill=black, draw,inner sep=0pt,
 minimum width=2pt] 
\tikzstyle{every pin edge}=[<-,shorten <=1pt]
\fill[color=gray!25] (10,5) -- (8,3) -- (4,1) -- cycle;
\fill[color=gray!50] (8,1) -- (8,3) -- (4,1) -- cycle;
\node (xk0)  [label=left:$\hat x^*_k(0) \equiv \hat x^*_k(1)$,draw] {}
node (xk2) at (4,1) [pin=above:$\hat x^*_k(2)$,draw]{}
node (y0) at (10,5) [label=below:$\hat y^*_0$,draw]{} 
node (y1) at (8,3) [label=right:$\hat y^*_1$]{} 
node (y2) at (8,1) [label=right:$\hat y^*_2$]{}
node (y3) at (9,0) {}
node (y4) at (11,-0.2) [label=above:$\hat y^*_j$]{}
node (y5) at (12,1){}
node (y6) at (11.5,4) {}; 
\draw {(y2) -- (y3) --(y4) -- (y5) -- (y6) -- (y0)}; 
\draw[dashed] {(y0) -- (y1) --(y2)}; 
\draw[dotted] {(xk0) -- (y0)};
\draw[dotted] {(xk0) -- (y1)};
\draw[dotted] {(xk0) -- (y2)};
\draw[dotted] {(xk0) -- (y3)};
\draw[dotted] {(xk0) -- (y4)};
\draw {(xk2) -- (y0)};
\draw {(xk2) -- (y1)};
\draw {(xk2) -- (y2)};
\draw {(xk2) -- (y3)};
\draw(11.5,2.5) node[draw=none,fill=none,label=left:$P_{k-1}$]{};
\draw(6,1.5) node[draw=none,fill=none,label=right:$T_1(2)$]{};
\draw(7.5,2.9) node[draw=none,fill=none,pin=above:$T_0(2)$]{};
\end{tikzpicture}
\caption{Constructing $P_k$ from $P_{k-1}$.}\label{fig:figure_P_k_cap_Z^2}
\end{figure}
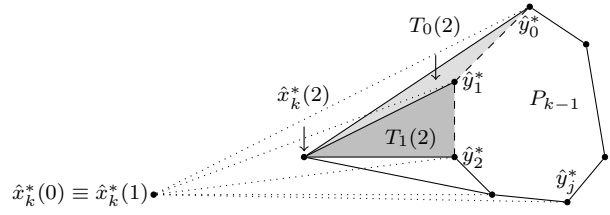

By Lemma~\ref{lem:P_k cap Z^2 = S_k}, the $k$-th best point
$\hat x_k^*$ can be assumed to be contained within
$[-B,B]^2\setminus\conv\{\hat x_1^*,\dots,\hat x_{k-1}^*\}$.
This property allows us to design a straightforward
algorithm to compute this point.
We first construct an inequality description of $\conv\{\hat
x_1^*,\dots,\hat x_{k-1}^*\}$, say
$\langle a_i, x \rangle \le b_i$ for $i \in I$ with
  $|I| < + \infty$.
Then \[[-B,B]^2\setminus\conv\{\hat x_1^*,\dots,\hat
x_{k-1}^*\}=\bigcup_{i \in
  I}\{x\in[-B,B]^2\;|\; \langle a_i, x \rangle>b_i\}.\]
As the feasible set
is described as a union of simple convex
sets, we could apply Algorithm~\ref{algo:mirror}
once for each of them.
However, instead of choosing this straightforward
approach one can do better: one
can avoid treating each element of this disjunction
separately by modifying Algorithm~\ref{algo:2d}
appropriately.

Suppose first that $k=2$.
To find the second best point, we apply
Algorithm~\ref{algo:2d} to the point $\hat x_1^*$ with the
following minor modification:
in Line~\ref{algo:2d:ip}, we replace $(\mathcal{P})$ with
the integer linear problem
$(\mathcal{P}'):\min\{\langle h,\hat y\rangle: \hat y\in\bar
T_k\cap\Z^2,\; \langle h,\hat y\rangle\ge\langle h,\hat
x_1\rangle+1\}$,
where $h \in \Z^2$ such that $\gcd(h_1,h_2) = 1$.
This prevents the algorithm from returning $\hat
x_1^*$ again.\label{page:tweak k=2}

Let $k\ge3$.
Let $\hat y^*_0,\dots,\hat y^*_{\nu-1},\hat y^*_\nu\equiv
\hat y^*_0$ denote the vertices of $P_{k-1}$, ordered
counterclockwise (they can be determined in
$\OO(k\ln(k))$ operations using the Graham Scan
\cite{Graham}).
Recall that the point we are looking for is not in $P_{k-1}$.

Let us call a triangle with a point $\hat y^*_i$ as vertex
and with a segment of the boundary of $[-B,B]^2$
as opposite side a \emph{search triangle}
(see Fig.~\ref{fig:triangulation2}:
every white triangle is a search triangle).
The idea is to decompose $[-B,B]^2\setminus P_{k-1}$ into
search triangles, then to apply
Algorithm~\ref{algo:2d} to these triangles
instead of $(\conv\{x,F_t\})_{t=1}^4$.

For each $0\leq i<\nu$, we define
$H_i:=\{y\in\R^2:\det(y-\hat y^*_i,\hat y^*_{i+1}-\hat
y^*_i)\geq 0\}$, so that
$H_i\cap P_{k-1} = \conv\{\hat y^*_i,\hat y^*_{i+1}\}$.
Consider the regions $R_i:=([-B,B]^2\cap H_i)\setminus
\inter H_{i-1}$. Note that $R_i$ contains only $\hat y^*_i$
and $\hat y^*_{i+1}$ as vertices of $P_{k-1}$.
Also, at most four of the $R_i$'s are no search
triangles.
If $R_i$ is such, we triangulate it into (at least two)
search triangles by inserting chords from $\hat y^*_i$ to
the appropriate vertices of $[-B,B]^2$.
%
%
%

\begin{figure}[ht]
\begin{minipage}[b]{0.50\linewidth}
\centering
\begin{tikzpicture}[scale=0.8]
\tikzstyle{every node}=[circle, fill=black, draw,inner sep=0pt,minimum width=2pt]
\draw{(0,0) -- (5,0) -- (5,5) -- (0,5) -- (0,0)};
\draw [dotted]{(1.5,1.833333) -- (0,1.3333333) };
\fill[color=gray!25] (2,3) -- (2,2) -- (3.5,2.5) -- (2,3);
\draw (2.2,2.5) node[draw=none,fill=none,label=right:$P_{k-1}$]{};
\node (v0) at (2,3) [label=above:$\hat y^*_0$]{} 
node (v1) at (2,2) [label=left:$\hat y^*_1$,fill]{} 
node (v2) at (3.5,2.5) [label=below:$\hat y^*_2$]{};
\draw {(v0) -- (2,0) };
\draw {(v1) -- (5,3) };
\draw {(v2) -- (0,3.6666666) };
\draw [->]{(1,1.6666666) -- (1.2,1.0666666) };
\draw (1.2,1.0666666) node[draw=none,fill=none,label=left:$H_1$]{};
\draw (4.5,0.5) node[draw=none,fill=none,label=left:$R_1$]{};
\end{tikzpicture}
\caption{Triangulation step 1.}\label{fig:triangulation1}
\end{minipage}
\hspace{0.5cm}
\begin{minipage}[b]{0.50\linewidth}
\centering
\begin{tikzpicture}[scale=0.8]
\tikzstyle{every node}=[circle, fill=black, draw,inner sep=0pt,minimum width=2pt]
\draw{(0,0) -- (5,0) -- (5,5) -- (0,5) -- (0,0)};
\fill[color=gray!25] (2,3) -- (2,2) -- (3.5,2.5) -- (2,3);
\draw (2.2,2.5) node[draw=none,fill=none,label=right:$P_{k-1}$]{};
\node (v0) at (2,3) [label=above:$\hat y^*_0$]{} 
node (v1) at (2,2) [label=left:$\hat y^*_1$]{} 
node (v2) at (3.5,2.5) [label=below:$\hat y^*_2$]{};
\draw {(v0) -- (2,0) (v0) -- (0,0)};
\draw {(v1) -- (5,3) (v1) -- (5,0)};
\draw {(v2) -- (0,3.6666666) (v2) -- (0,5) (v2) -- (5,5)};
\end{tikzpicture}
\caption{Triangulation step 2.}\label{fig:triangulation2}
\end{minipage}
\end{figure}

Note that a search triangle can contain two
or more integer points of $P_{k-1}$.
In order to prevent us from outputting one of those, we need
to perturb the search triangles slightly before
using them in Algorithm~\ref{algo:2d}.
Let $T = \conv\{\hat y^*_i,b_1,b_2\}$ be one of the search
triangles, with $b_1,b_2$ being points of the boundary of
$[-B,B]^2$.
The triangle $T$ might contain $\hat y^*_{i+1}$, say $\hat
y^*_{i+1}\in\conv\{\hat y^*_i,b_1\}$, a point we need to
exclude from $T$.
We modify $b_1$ slightly by replacing it with
$(1-\varepsilon) b_1 + \varepsilon b_2$ for an appropriate
positive $\varepsilon>0$ whose encoding length is $\OO(\ln(B))$. 

So, we apply Algorithm~\ref{algo:2d} with all these modified
search triangles instead of
$\conv\{x,F_1\},\ldots,\conv\{x,F_4\}$. A simple
modification of Line~\ref{algo:2d:ip} allows us to avoid the
point $\hat y^*_i$ for $\hat z$: we just need to replace the
linear integer problem $(\mathcal{P})$ with  $\min\{\langle
h,\hat y\rangle: \hat y\in\bar T_k\cap\Z^2,\; \langle h,\hat
y\rangle\ge\langle h, \hat y^*_i \rangle+1\}$,
where $h \in \Z^2$ such that $\gcd(h_1,h_2) = 1$.
Then, among the feasible integer points found, we return the
point with smallest objective value.

\begin{corollary}\label{cor:kth_point}
Let $f : \R^2 \to \R$ and $g_i : \R^2 \to \R$ with
$i=1,\dots,m$ be convex functions.
Let $\hat x_1^*,\dots,\hat x_{k-1}^*$ be the $k-1$ best
points for $\min \{f(\hat x) :
\hat x \in S \cap [-B,B]^2 \cap \Z^2\}$. 
Then, in a number of evaluations of $f$ and $g_1,\dots,g_m$
that is polynomial in $\ln(B)$ and in $k$, one can either
find
\begin{enumerate}[(a)]
  \item a $k$-th best point, $\hat x_k^*$, or
  \item show that there is no such point.
\end{enumerate}\end{corollary}

\section{Extensions and applications to the general setting} \label{sec.2D+d}

In this section, we extend our algorithm for solving
two-dimensional integer convex optimization problems in
order to solve more general mixed-integer convex problems.
The first extension concerns mixed-integer
convex problems with two integer variables and
$d$ continuous variables.
For those, we first need results about problems
with only one integer variable. We derive these results in
Subsection~\ref{sec:MIP1D} where we propose a variant of
the well-known golden search method that deals with convex
functions whose value is only known approximately.
To the best of our knowledge, this variant is new.

In Subsection~\ref{sec:MIP2D}, we build an
efficient method for solving mixed-integer convex problems
with two integer and $d$ continuous variables
and propose an extension of Corollary~\ref{cor:kth_point}.
This result itself will be used as a subroutine to design a
finite-time algorithm for mixed-integer convex
problems in $n$ integer and $d$ continuous
variables in Subsection~\ref{sec:MIPgen}.

In this section, the problem of interest is \eqref{eq:general_MIP_f}:
\[
\min\{f(\hat x,y):g_i(\hat x,y)\leq 0 \textrm{ for } 1\leq i\leq m, (\hat x,y)\in\Z^n\times\R^d\}
\] 
with a few mild simplifying assumptions. 
We define the function
$$ g: \R^n \to \R , \quad x \mapsto g(x) :=
   \min_{y\in\R^d} \max_{1\leq i\leq m} g_i(x,y).$$
We assume that this minimization in $y$ has a solution for
every $x\in\R^n$, so as to make the function $g$ convex.
Let $S:=\{(x,y)\in\R^{n+d}:g_i(x,y)\leq 0 \textrm{ for }
1\leq i\leq m\}$. We assume that the function $f$ has a
finite spread $\max\{f(x,y)-f(x',y'):(x,y),(x',y')\in S\}$
on $S$ and that we know an upper bound $V_f$ on that
spread.
Observe that, by Lipschitz continuity of $f$ and the
assumption that we optimize over $[-B,B]^n$, it follows
$V_f \le 2 \sqrt{n} B L$.
Finally, we assume that the partial minimization
function:
$$ \phi:\R^n \to \R \cup \{+\infty\}, \quad x \mapsto
   \phi(x) := \min\{f(x,y):(x,y)\in S\} $$
is convex.
As for the function $g$, this property can be achieved
e.g.~if for every $x\in\R^n$ for which $g(x)\leq 0$ there
exists a point $y$ such that $(x,y)\in S$ and
$\phi(x) = f(x,y)$. 

Our approach is based on the following well-known identity:
\[
\min\{f(\hat x,y):(\hat x,y)\in S\cap(\Z^n\times\R^d)\} =
\min\{\phi(\hat x): g(\hat x)\leq 0,\ \hat x\in\Z^n\}.
\]
For instance, when $n=2$, we can use the techniques developed in the previous section on $\phi$ to implement the improvement oracle for $f$. However, we cannot presume to know exactly the value of $\phi$, as it results from a minimization problem. We merely assume that, for a known accuracy $\gamma>0$ and for every $x\in\dom\phi$ we can determine a point $y_x$ such that $(x,y_x)\in S$ and 
$f(x,y_x)-\gamma\leq \phi(x)\leq f(x,y_x)$.
Determining $y_x$ can be, on its own, a non-trivial
optimization problem. Nevertheless, it is a convex problem
for which we can use the whole machinery of standard Convex
Programming (see e.g. \cite{Nesterov_Nemirovski_book,Toint_book,Nesterov-Book04} and references therein.).

Since we do not have access to exact values of $\phi$, we cannot hope for an exact oracle for the function $\phi$, let alone for $f$. The impact of the accuracy $\gamma$ on the accuracy of the oracle is analyzed in the next subsections.

\subsection{Mixed-integer convex problems with one integer variable}\label{sec:MIP1D}

The Algorithm~\ref{algo:2d} uses as indispensable tools the
bisection method for solving two types of problems:
minimizing a convex function over the integers
of an interval, and finding, in a given interval,
a point that belongs to a level set of a convex function.
In this subsection, we show how to adapt the bisection
methods for mixed-integer problems.
It is well-known that the bisection method is the fastest
for minimizing univariate convex functions over a finite
segment (\cite[Chapter 1]{Nemirovski_Course_Convex}).

Let $a,b\in\R$, $a<b$, and $\varphi:[a,b]\to\R$ be a convex function to minimize on $[a,b]$ and/or on the integers of $[a,b]$, such as the function $\phi$ in the preamble of this Section \ref{sec.2D+d} when $n=1$. Assume that, for every $t\in [a,b]$, we know a number $\tilde \varphi(t)\in[\varphi(t),\varphi(t)+\gamma]$.
In order to simplify the notation, we scale the problem so that $[a,b] \equiv [0,1]$. The integers of $\aff\{a,b\}$ are scaled to a set of points of the form $t_0+\tau\Z$ for a $\tau>0$. Of course, the spread of the function $\varphi$ does not change, but its Lipschitz constant does, and achieving the accuracy $\gamma$ in its evaluation must be reinterpreted accordingly.

%
%
%
In the sequel of this section, we fix
$0 \le \lambda_0 < \lambda_1 \le 1$.

\begin{lemma}\label{lem:bissection1}
Under our assumptions, the following statements hold.
\begin{enumerate}[(a)]
  \item If $\tilde \varphi(\lambda_0)\leq \tilde
    \varphi(\lambda_1)-\gamma$, then $\varphi(\lambda)\geq
    \tilde \varphi(\lambda_0)$ for all $\lambda \in
    [\lambda_1,1]$.
  \item If $\tilde \varphi(\lambda_0)\geq \tilde
    \varphi(\lambda_1)+\gamma$, then $\varphi(\lambda)\geq
    \tilde \varphi(\lambda_1)$ for all $\lambda \in [0,\lambda_0]$.
\end{enumerate}
\end{lemma}

\begin{proof}
\smartqed
We only prove Part (a) as the proof of Part (b) is
symmetric. Thus, let us assume that
$\tilde \varphi(\lambda_0)\leq \tilde \varphi(\lambda_1)-\gamma$.
Then there exists $0<\mu\leq 1$ for which $\lambda_1 = \mu \lambda + (1-\mu)\lambda_0$. Convexity of $\varphi$ allows us to write:
$$ \tilde \varphi(\lambda_0)  \leq\tilde
\varphi(\lambda_1)-\gamma\leq \varphi(\lambda_1) \leq
\mu\varphi(\lambda) + (1-\mu)\varphi(\lambda_0) \leq \mu
\varphi(\lambda) + (1-\mu)\tilde \varphi(\lambda_0), $$
implying $\tilde \varphi(\lambda_0)\leq \varphi(\lambda)$ as
$\mu>0$.
Fig.~\ref{fig:symmetric} illustrates the proof graphically.
\qed
\end{proof}

\begin{figure}[ht]
\centering
\begin{tikzpicture}[scale=7]
\tikzstyle{every node}=[circle, fill=black, draw,inner sep=0pt,
minimum width=2pt];
\node (a) [label=above:$0$,draw]{}
node (b) at (1,0)[label=above:$1$,draw]{}
node (L1) at (0.618033988,0)[label=above:$\lambda_1$,draw]{}
node (L0) at (0.381966011,0)[label=above:$\lambda_0$,draw]{}
node (L0g) at (0.381966011,-0.15)[draw] {}
node (L1g) at (0.618033988,-0.07)[draw] {}
node (L1f) at (0.618033988,-0.12)[draw] {}
node (low1) at (1,-0.07145898) [draw] {};
\draw  {(a) -- (b)};
\draw  {(L1g) -- (L1f)};
\draw [gray] {(L0g) -- (L1f)};
\draw [thick] {(L1f) -- (low1)};
\draw [dotted,thin] {(L0g) -- (0,-0.15)};
\draw [dashed,thin] {(1,-0.15) -- (L0g)};
\draw [dotted,thin] {(L1g) -- (0,-0.07)};
\draw (0,-0.15) node[draw=none,fill=none,label=left:$\tilde \varphi(\lambda_0)$] {};
\draw (0,-0.07) node[draw=none,fill=none,label=left:$\tilde \varphi(\lambda_1)$] {};
\draw [<->] {(0.15,-0.12) -- (0.15,-0.07)};
\draw [dotted,thin] {(L1f) -- (0.15,-0.12)};
\draw [dotted,thin] {(L1) -- (L1f)};
\draw [dotted,thin] {(L0) -- (L0g)};
\draw (0.15,-0.095) node[draw=none,fill=none,label=left:$\gamma$] {};
\end{tikzpicture}
\caption{Lemma \ref{lem:bissection1}: the bold line
  represents a lower bound on $\varphi$ in Part (a).}
\label{fig:symmetric}
\end{figure}
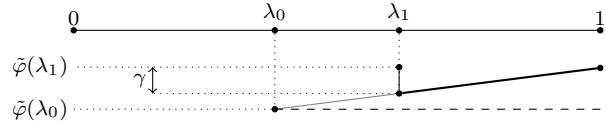

If one of the conditions in Lemma~\ref{lem:bissection1} is
satisfied, we can remove from the interval $[0,1]$ either
$[0,\lambda_0[$ or $]\lambda_1,1]$.
To have a symmetric effect of the algorithm in either case,
we set $\lambda_1 := 1-\lambda_0$, forcing $\lambda_0$ to be
smaller than $\frac{1}{2}$.
In order to recycle our work from iteration to iteration, we
choose $\lambda_1:=\frac{1}{2}(\sqrt{5}-1)$, as in the
golden search method:
if we can eliminate, say, the interval $]\lambda_1,1]$ from
$[0,1]$, we will have to compute in the next
iteration step an approximate value of the objective
function at $\lambda_0 \lambda_1$ and $\lambda_1^2$.
The latter happens to equal $\lambda_0$ when
$\lambda_1 = \frac{1}{2}(\sqrt{5}-1)$.

It remains to define a strategy when neither of the
conditions in Lemma~\ref{lem:bissection1} is
satisfied.
In the lemma below, we use the values for
$\lambda_0,\lambda_1$ chosen above.

\begin{lemma}\label{lem:bissectionterm}
Assume that
$\tilde \varphi(\lambda_1)-\gamma<\tilde \varphi(\lambda_0)< \tilde\varphi(\lambda_1)+\gamma$.
We define:
\begin{align*}
  \lambda_{0+} & := (1-\lambda_0) \cdot \lambda_0+\lambda_0
  \cdot \lambda_1 = 2\lambda_0 \lambda_1, \\
  \lambda_{1+} & := (1-\lambda_1) \cdot \lambda_0+\lambda_1
  \cdot \lambda_1 = 1-2\lambda_0 \lambda_1.
\end{align*}
If $\min\{\tilde \varphi(\lambda_{0+}),\tilde
\varphi(\lambda_{1+})\} \leq \min\{\tilde \varphi(\lambda_{0})-\gamma,\tilde \varphi(\lambda_{1})-\gamma\}$,
then $\varphi(t)\geq \min\{\tilde\varphi(\lambda_{0+}),$
$\tilde\varphi(\lambda_{1+})\}$ for all $t\in[0,1]\setminus [\lambda_0,\lambda_1]$.
Otherwise, it holds that $\min\{\tilde \varphi(\lambda_{0}),\tilde \varphi(\lambda_1)\}\leq \min\{\varphi(t):t\in[0,1]\}+(\kappa-1)\gamma$,
where $\kappa := \frac{2}{\lambda_0}\approx 5.236$.
\end{lemma}

\begin{proof}
\smartqed
The first conclusion follows immediately from
Lemma~\ref{lem:bissection1}.
The second situation involves a tedious enumeration,
summarized in Fig.~\ref{fig:figure5}.
We assume, without loss of generality, that $\tilde
\varphi(\lambda_{0})\leq\tilde\varphi(\lambda_{1})$.
The bold lines in Fig.~\ref{fig:figure5} represent a lower bound on the value of the function $\varphi$. We show below how this lower bound is constructed and determine its lowest point. In fact, this lower bound results from six applications of a simple generic inequality \eqref{eq:generic_inequality} that we establish below, before showing how we can particularize it to different segments of the interval $[0,1]$.

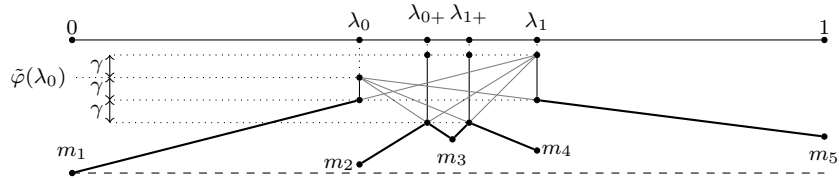
\begin{figure}[ht]
\centering
\begin{tikzpicture}[scale=10]
\tikzstyle{every node}=[circle, fill=black, draw,inner sep=0pt,minimum width=2pt];
\node (a) [label=above:$0$,draw]{}
node (b) at (1,0)[label=above:$1$,draw]{}
node (L1) at (0.618033988,0)[label=above:$\lambda_1$,draw]{}
node (L0) at (0.381966011,0)[label=above:$\lambda_0$,draw]{}
node (L1plus) at (0.527864045,0)[label=above:$\lambda_{1+}$,draw]{}
node (L0plus) at (0.472135955,0)[label=above:$\lambda_{0+}$,draw]{}
node (L0g) at (0.381966011,-0.05)[draw] {}
node (L0f) at (0.381966011,-0.08)[draw] {}
node (L1g) at (0.618033988,-0.02)[draw] {}
node (L1f) at (0.618033988,-0.08)[draw] {}
node (L0plusg) at (0.472135955,-0.02)[draw] {}
node (L0plusf) at (0.472135955,-0.11)[draw] {}
node (L1plusg) at (0.527864045,-0.02)[draw] {}
node (L1plusf) at (0.527864045,-0.11)[draw] {}
node (low1) at (0,-0.17708204) [label=above:$m_1$,draw] {}
node (low2) at (0.381966011,-0.165623059) [label=left:$m_2$,draw] {}
node (low3) at (0.505572809,-0.132249224) [label=below:$m_3$,draw] {}
node (low4) at (0.618033988,-0.147082039) [label=right:$m_4$,draw] {}
node (low5) at (1,-0.12854102) [label=below:$m_5$,draw] {};
\draw  {(L0g) -- (L0f)};
\draw  {(L1g) -- (L1f)};
\draw  {(L0plusg) -- (L0plusf)};
\draw  {(L1plusg) -- (L1plusf)};
\draw [gray] {(L1g) -- (L0f)};
\draw [thick] {(L0f) -- (low1)};
\draw [gray] {(L1g) -- (L0plusf)};
\draw [thick] {(L0plusf) -- (low2)};
\draw [gray] {(L0g) -- (L0plusf)};
\draw [thick] {(L0plusf) -- (low3)};
\draw [gray] {(L1g) -- (L1plusf)};
\draw [thick] {(L1plusf) -- (low3)};
\draw [gray] {(L0g) -- (L1plusf)};
\draw [thick] {(L1plusf) -- (low4)};
\draw [gray] {(L0g) -- (L1f)};
\draw [thick] {(L1f) -- (low5)};
\draw  {(a) -- (b)};
\draw [dotted,thin] {(L0g) -- (0,-0.05)};
\draw [<->] {(0.05,-0.05) -- (0.05,-0.08)};
\draw [<->] {(0.05,-0.08) -- (0.05,-0.11)};
\draw [<->] {(0.05,-0.05) -- (0.05,-0.02)};
\draw (0,-0.05) node[draw=none,fill=none,label=left:$\tilde \varphi(\lambda_0)$] {};
\draw (0.05,-0.065) node[draw=none,fill=none,label=left:$\gamma$] {};
\draw (0.05,-0.095) node[draw=none,fill=none,label=left:$\gamma$] {};
\draw (0.05,-0.035) node[draw=none,fill=none,label=left:$\gamma$] {};
 \draw [dotted,thin] {(L1g) -- (0.05,-0.02)};
 \draw [dotted,thin] {(L1f) -- (0.05,-0.08)};
 \draw [dotted,thin] {(L1plusf) -- (0.05,-0.11)};
 \draw [dashed,thin] {(0,-0.17708204) -- (1,-0.17708204)};
 \draw [dotted,thin] {(L1g) -- (L1)};
 \draw [dotted,thin] {(L1plusg) -- (L1plus)};
 \draw [dotted,thin] {(L0g) -- (L0)};
 \draw [dotted,thin] {(L0plusg) -- (L0plus)};
\end{tikzpicture}
\caption{Approximate bisection: bold lines represent a lower bound on $\varphi$ in the termination case.} \label{fig:figure5}
\end{figure}

Let $0<t<1$ and let $u,v\in\{\lambda_0,\lambda_{0+},\lambda_{1+},\lambda_1\}$. Suppose that we can write 
$v=\mu t+(1-\mu) u$ for a $\mu\in\ ]\mu_0,1]$ with
$\mu_0>0$.
If we can find constants $\gamma_-, \gamma_+ \ge 0$ that
satisfy
\[\varphi(v)+\gamma_+\geq\tilde \varphi(\lambda_0)\geq
\varphi(u)-\gamma_-\] then we can infer:
\[
\mu \varphi(t) + (1-\mu)(\tilde \varphi(\lambda_0)+\gamma_-)
\geq \mu \varphi(t) + (1-\mu)\varphi(u)
\geq \varphi(v)
\geq \tilde \varphi(\lambda_0)-\gamma_+,
\]
and thus:
\begin{equation}\label{eq:generic_inequality}
 \varphi(t) -\tilde \varphi(\lambda_0) \geq \gamma_- -\frac{\gamma_++\gamma_-}{\mu}\geq
\gamma_--\frac{\gamma_++\gamma_-}{\mu_0}.
\end{equation}
\begin{enumerate}
\item If $t\in\ ]0,\lambda_0]$, we can take $u:=\lambda_1$ and $v:=\lambda_0$, giving $\mu_0 = 1-\frac{\lambda_0}{\lambda_1} = \lambda_0$. Then $\gamma_- = \gamma_+ = \gamma$, and $\varphi(t) -\tilde \varphi(\lambda_0) \geq-\gamma(\frac{2}{\lambda_0}-1)$.
\item If $t\in\ ]\lambda_1,1[$, we choose $u:=\lambda_0$ and
  $v:=\lambda_1$, and by symmetry with the previous case we
  obtain $\mu_0 = \lambda_0$. Now, $\gamma_- = 0$ and $\gamma_+ = \gamma$, yielding a higher bound than in the previous case.
\item Suppose $t\in\ ]\lambda_0,\lambda_{0+}]$. Then with $u:=\lambda_1$ and $v:=\lambda_{0+}$, we get $\mu_0=\frac{\lambda_1-\lambda_{0+}}{\lambda_1-\lambda_0} = \lambda_1$, $\gamma_-=\gamma$, $\gamma_+=2\gamma$, giving as lower bound $-\gamma(\frac{3}{\lambda_1}-1)$, which is higher than the first one we have obtained.
\item Symmetrically, let us consider $t\in\ ]\lambda_{1+},\lambda_1]$. With $u:=\lambda_0$ and $v:=\lambda_{1+}$, we obtain also $\mu_0 = \lambda_1$. As $\gamma_- = 0$ and $\gamma_+=2\gamma$, the lower bound we get is larger than the one in the previous item.
\item Set $\lambda':=\frac{1}{5}(2\lambda_{0+}+3\lambda_{1+})$. If $t\in\ ]\lambda_{0+},\lambda']$, we can use $u:=\lambda_0$ and $v:=\lambda_{0+}$, so that $\mu_0 = \frac{\lambda_{0+}-\lambda_0}{\lambda'-\lambda_0} = 5\lambda_0^2$, $\gamma_-=0$, and $\gamma_+=2\gamma$. Thus, the lower bound is evaluated as $-\frac{2\gamma}{5\lambda_0^2}$, which is higher than any of the bounds we have obtained so far.
\item Finally, if $t\in\ ]\lambda',\lambda_{1+}]$, we take $u:=\lambda_1$ and $v:=\lambda_{1+}$, so that $\gamma_-=\gamma$,  $\gamma_+=2\gamma$, and $\mu_0 = \frac{\lambda_1-\lambda_{1+}}{\lambda_1-\lambda'} = \frac{5\lambda_0}{2+\lambda_0}$. Hence, we get $-\gamma(\frac{3(2+\lambda_0)}{5\lambda_0}-1) = -\frac{2\gamma}{5\lambda_0^2}$ for the lower bound, just as in the previous item.
\end{enumerate}
So, the lower bound for $\varphi(t)-\tilde\varphi(\lambda_0)$ on $[0,1]$ can be estimated as $-\gamma(\frac{2}{\lambda_0}-1)\approx -4.236\gamma$.
\qed
\end{proof}

In the proof of the following proposition, we
present an algorithm that returns a point $x \in [0,1]$
whose function value $\varphi(x)$ is close to
$\min\{\varphi(t):t\in[0,1]\}$.

\begin{proposition}\label{prop:bisection_continuous}
There exists an algorithm that finds a point $x\in[0,1]$ for
which $\tilde\varphi(x)-(\kappa-1)
\gamma\leq\min\{\varphi(t):t\in [0,1]\}\leq
\varphi(x)$ in
at most $2+\left\lceil\ln\left(\frac{(\kappa-1)\gamma}{V_\varphi}\right)/\ln(\lambda_1)\right\rceil$ evaluations of $\tilde \varphi$, where $V_\varphi$ is the spread of $\varphi$ on $[0,1]$.
\end{proposition}

\begin{proof}
\smartqed
We start with the interval $[0,1]$ and by evaluating
$\tilde \varphi$ at $\lambda_0$ and $\lambda_1$.
If one of the two conditions in Lemma~\ref{lem:bissection1}
is satisfied, we can shrink the interval by a factor of
$\lambda_0 \approx 38\%$ since it suffices to continue
either with the interval $[0,\lambda_1]$ or with
$[\lambda_0,1]$.
If not, then Lemma~\ref{lem:bissectionterm} applies:
if the first condition stated in
Lemma~\ref{lem:bissectionterm} is met, then it suffices to
continue with the interval $[\lambda_0,\lambda_1]$ so as to
shrink the starting interval by a factor of
$2\lambda_0\approx 76\%$.
Otherwise, any $x \in [\lambda_0,\lambda_1]$ satisfies the
requirement of the lemma and we can stop the algorithm.
Therefore, either the algorithm stops or we shrink the
starting interval by a factor of at least $\lambda_0$.
Iterating this procedure, it follows that --- if the
algorithm does not stop --- at every step the length of the
remaining interval is at most $\lambda_1$ times the length
of the previous interval.
Moreover, by the choice of $\lambda_0$, the function
$\tilde \varphi$ is
evaluated in two points at the first step, and in only one
point as from the second step in the algorithm.
So, at iteration $k$, we have performed at most $2+k$
evaluations of $\tilde \varphi$.

By construction, the minimum $t^*$ of $\varphi$ lies in the
remaining interval $I_k$ of iteration $k$.
Also, the value of $\varphi$ outside $I_k$ is higher than
the best value found so far, say $\tilde\varphi(\bar t_k)$.
Finally, the size of $I_k$ is bounded from above by
$\lambda_1^k$.
Consider now the segment $I(\lambda):=(1-\lambda)t^*+\lambda
[0,1]$, of size $\lambda$.
Observe that for every $\lambda$ such that
$1\geq\lambda>\lambda_1^k$, the interval $I(\lambda)$
contains a point that is not in $I_k$.
Therefore, 
\begin{eqnarray*}
\tilde\varphi(\bar t_k)  \leq  \max\{\varphi(t):t\in I(\lambda)\} &\leq& (1-\lambda)\varphi(t^*)+\lambda \max\{\varphi(t):t\in [0,1]\}\\
&\leq&
  (1-\lambda)\varphi(t^*)+\lambda (V_\varphi+\varphi(t^*)). 
\end{eqnarray*}
Hence $\tilde \varphi(\bar t_k) -
\varphi(t^\ast) \le \lambda V_\varphi$, and, by taking
$\lambda$ arbitrarily close to $\lambda_1^k$, we get
$\tilde\varphi(\bar t_k)-\varphi(t^*)\leq \lambda^k_1
V_\varphi$.
If the algorithm does not end prematurely, we need at most
$\left\lceil\ln\left(\frac{(\kappa-1)\gamma}{V_\varphi}\right)/\ln(\lambda_1)\right\rceil$
iterations to make
$\lambda^k_1 V_\varphi$ smaller than
$(\kappa-1)\gamma$.
\qed
\end{proof}

\begin{remark}
If we content ourselves with a coarser precision $\eta\geq(\kappa-1)\gamma$, we merely need $\OO(\ln(V_\varphi/\eta))$ evaluations of $\tilde\varphi$.$\hfill\diamond$
\end{remark}

It is now easy to extend this procedure to minimize a convex
function approximately over the integers of an interval
$[a,b]$, or, using our simplifying scaling, over
$(t_0+\tau\Z)\cap [0,1]$ for given $t_0\in\R$ and $\tau>0$.

\begin{proposition}\label{prop:bisection_integer}
There exists an algorithm that finds a point $\hat x\in(t_0+\tau\Z)\cap [0,1]$ for which:
\[
\tilde \varphi(\hat x)-\kappa\gamma\leq \min\{\varphi(\hat t):\hat t\in(t_0+\tau\Z)\cap [0,1]\}\leq \varphi(\hat x)
\]
in less than
\[
\min\left\{4+\left\lceil\frac{\ln((\kappa-1)\gamma/V_\varphi)}{\ln(\lambda_1)}\right\rceil,5+\left\lceil\frac{\ln(\tau)}{\ln(\lambda_1)}\right\rceil\right\}
\]
evaluations of $\tilde \varphi$, where $V_\varphi$ is the spread of $\varphi$ on $[0,1]$.
\end{proposition}
\begin{proof}
\smartqed We denote in this proof the points in $(t_0+\tau\Z)$ as \emph{scaled integers}. To avoid a trivial situation, we assume that $[0,1]$ contains at least two such scaled integers. 

Let us use the approximate bisection method described in
the proof of Proposition~\ref{prop:bisection_continuous} until the remaining interval has a size smaller than $\tau$, so that it contains at most one scaled integer. Two possibilities arise: either the algorithm indeed finds 
such a small interval $I_k$, or it finishes prematurely, with a remaining interval $I_k$ larger than $\tau$. 

In the first case, which requires at most
$2+\lceil\ln(\tau)/\ln(\lambda_1)\rceil$
evaluations of $\tilde\varphi$, we know that $I_k$ contains
the continuous minimizer of $\varphi$.
Hence, the actual minimizer of $\varphi$ over
$(t_0+\tau\Z)\cap [0,1]$ is among at most three scaled
integers, namely the possible scaled integer in $I_k$, and,
at each side of $I_k$, the possible scaled integers that are
the closest to $I_k$.
By convexity of $\varphi$, the best of these three points, say $\hat x$, satisfies
$\tilde\varphi(\hat x)-\gamma\leq \varphi(\hat x) =
\min\{\varphi(\hat t):\hat t\in(t_0+\tau\Z)\cap [0,1]\}$.

In the second case, we have an interval $I_k\subseteq [0,1]$
and a point $\bar t_k$ that fulfill $\tilde\varphi(\bar
t_k)\leq\min\{\varphi(t): t\in[0,1]\}+(\kappa-1)\gamma$,
which was determined within at most $2+\left\lceil\frac{\ln((\kappa-1)\gamma/V_\varphi)}{\ln(\lambda_1)}\right\rceil$ evaluations of $\tilde\varphi$. Consider the two scaled integers $\hat t_-$ and $\hat t_+$ that are the closest from $\bar t_k$. One of these two points constitutes an acceptable output for our algorithm. Indeed, suppose first that $\min\{\tilde \varphi(\hat t_-), \tilde\varphi(\hat t_+)\}\leq\tilde \varphi(\bar t_k)+\gamma$. Then:
\begin{eqnarray*}
\min\{\tilde \varphi(\hat t_-),\tilde \varphi(\hat t_+)\}\leq\tilde \varphi(\bar t_k)+\gamma\leq \min\{\varphi(t):t\in[0,1]\} + \kappa\gamma,
\end{eqnarray*}
and we are done. Suppose that $\min\{\tilde \varphi(\hat
t_-), \tilde\varphi(\hat t_+)\}>\tilde \varphi(\bar
t_k)+\gamma$ and that there exists a scaled integer $\hat t$
with $\varphi(\hat t)<\min\{\varphi(\hat t_-), \varphi(\hat
t_+)\}$. Without loss of generality, let $\hat t_-\in\conv\{\hat t,\bar t_k\}$, that is $\hat t_- = \lambda \hat t+(1-\lambda)\bar t_k$, with $0\leq\lambda < 1$.
We have by convexity of $\varphi$:
\[
\varphi(\hat t_-)\leq \lambda \varphi(\hat t)+(1-\lambda)\varphi(\bar t_k)<\lambda \varphi(\hat t_-)+(1-\lambda)(\tilde\varphi(\hat t_-)-\gamma),
\]
which is a contradiction because $\lambda<1$ and $\tilde
\varphi(\hat t_-) -\gamma\leq\varphi(\hat t_-)$.
So, it follows that $\varphi(\hat t) \ge
\min\{\varphi(\hat t_-),\varphi(\hat t_+)\}$ for every
$\hat t\in(t_0+\tau\Z)\cap [0,1]$, proving the statement.
\qed
%
\end{proof}

In the the following we extend the above results to the problem $\min\{\varphi(t):t\in [0,1],\; g(t)\leq 0\}$, where $g:[0,1]\to\R$ is a convex function with a known spread $V_g$. 
In the case that we have access to exact values of $g$, an approach for attacking the problem would be the following:
we first determine whether there exists an element $\bar{t}\in[0,1]$ with $g(\bar{t})\leq 0$.
If $\bar{t}$ exists, we determine the exact bounds $t_-$ and $t_+$ of the interval  $\{t\in [0,1], g(t)\leq 0\}$.
Then we minimize the function $f$ over $[t_-,t_+]$.

The situation where we do not have access to exact values of $g$ or where we cannot determine the feasible interval $[t_-,t_+]$  induces some technical complications. We shall not investigate them in this paper, except in the remaining of this subsection in order to appreciate the modification our method needs in that situation: let us assume, that we have only access to a value $\tilde g(t)\in [g(t),g(t)+\gamma]$. 
In order to ensure that the constraint $g$ is well-posed we make an additional assumption: either $\{t\in [0,1]:|g(t)|\leq \gamma\}$ is  empty, or the quantity $\min\{|g'(t)|:g'(t)\in\partial g(t), |g(t)|\leq \gamma\}$ is non-zero, and even reasonably large. 
This ensures that the (possibly empty) \mbox{$0$-level} set of $g$ is known with enough accuracy.
We denote by $\theta>0$ a lower bound on this minimum, and for simplicity assume that $\theta = 2^N\gamma$ for a suitable $N\in\N$.


Our strategy proceeds as follows. First we determine whether there exists a point $\bar t\in [0,1]$ for which $g(\bar t)<0$ by applying the minimization procedure described in Proposition \ref{prop:bisection_continuous}. 
If this procedure only returns nonnegative values, we can conclude after at most $2+\lceil\ln((\kappa-1)\gamma/V_g)/\ln(\lambda_1)\rceil$ evaluations of $\tilde{g}$ that 
$g(t)\geq -(\kappa-1)\gamma$, in which case we declare that we could not locate any feasible point in $[0,1]$.

Otherwise, if  we find a point $\bar t\in[0,1]$ with $\tilde g(\bar t)<0$,
we continue and compute approximate bounds $t_-$ and $t_+$ of the interval  $\{t\in [0,1], g(t)\leq 0\}$.
For that, we assume $\tilde g(0),\tilde g(1)\ge0$.
By symmetry, we only describe how to construct $t_-$ such that $\tilde g(t_-)\leq 0$
and $g(t_--\eta)\geq 0$ for an $\eta>0$ reasonably small. 
Note that $g(t)\leq 0$ on $[t_-,\bar t]$ by convexity of $g$.

%

In order to compute $t_-$, we adapt the standard bisection method for finding a root of a function.
Note that the function $\tilde g$ might not have any root as it might not be continuous. 
Our adapted method constructs a decreasing sequence of intervals $[a_k,b_k]$ such that $\tilde g(a_k) > 0$, $\tilde g(b_k)\leq 0$, and $b_{k+1}-a_{k+1}= \frac{1}{2}(b_k-a_k)$. 
If $\tilde g(a_k) >\gamma$, we know that $g$ is positive on $[0,a_k]$, and we know that there is a root of $g$ on $[a_k,b_k]$. 
Otherwise, if $0<\tilde g(a_k)\leq\gamma$ and that the interval $[a_k,b_k]$ has a length larger or equal to 
$\frac{\gamma}{\theta}$. 
Given the form of $\theta$, we know
that $k\leq N$. 
We claim that for every $0\leq t\leq
\min\{0,a_k-\frac{\gamma}{\theta}\}$ we have
$g(t)\geq 0$, so that we can take
$\eta:=2\frac{\gamma}{\theta}$ and $t_-:=b_N$.
Indeed, assume that $g'(a_k) \ge \theta$, then
\[
\tilde g(b_k)\geq g(b_k)\geq g(a_k) + g'(a_k)(b_k-a_k) > -\gamma + \theta\cdot\frac{\gamma}{\theta}\geq 0
\]
giving a contradiction, so we must have $g'(a_k) \le -\theta$.
We can exclude the case where $t$ can only be $0$. As claimed, we have \[g(t)\geq g(a_k) + g'(a_k)(t-a_k) \geq  -\gamma + \theta(a_k-t)\geq 0\] as $\frac{\gamma}{\theta}\leq a_k-t$. This takes $\left\lceil\ln(\frac{\gamma}{\theta}) / \ln(\frac{1}{2})\right\rceil$ evaluations of $\tilde g$.

Summarizing this, we just sketched the proof of the following corollary.

\begin{corollary}\label{cor:bisection_constraints}
There exists an algorithm that solves
approximately
$\min\{\varphi(t) : t \in [0,1], g(t) \le 0\}$, in the sense that it finds, if they exist, three points $0\leq t_-\leq x\leq t_+\leq 1$ with:
\begin{enumerate}[(a)]
 \item $g(t)\leq \tilde g(t)\leq 0$ for every $t\in[t_-,t_+]$,
 \item if $t_-\geq2\frac{\gamma}{\theta}$, then $g(t)\geq 0$ for every $t\in [0,t_--2\frac{\gamma}{\theta}]$,
 \item if $t_+\leq 1-2\frac{\gamma}{\theta}$, then $g(t)\geq 0$ for every $t\in [t_++2\frac{\gamma}{\theta},1]$,
 \item $\tilde\varphi(x)\leq\min\{\varphi(t):t\in[t_-,t_+], \;g(t)\le0\}+(\kappa-1)\gamma$
\end{enumerate}
within at most
$3+\left\lceil\frac{\ln((\kappa-1)\gamma/V_g)}{\ln(\lambda_1)}\right\rceil
  +2\left\lceil\frac{\ln(\gamma/\theta)}{\ln(1/2)}\right\rceil$
  evaluations of $\tilde g$ and at most
  $2+\left\lceil\frac{\ln((\kappa-1)\gamma/V_\varphi)}{\ln(\lambda_1)}\right\rceil$
evaluations of $\tilde\varphi$.
\end{corollary}

As stressed before above, we assume from now on that we can compute exactly the roots of the function $g$ on a given interval, so that the segment $[t_-,t_+]$ in Corollary \ref{cor:bisection_constraints} is precisely our feasible set. This situation occurs e.g. in mixed-integer convex optimization with one integer variable when the feasible set $S\subset\R\times\R^d$ is a polytope.

\begin{remark}
In order to solve problem \eqref{eq:general_MIP_f} with one integer variable, we can extend Proposition \ref{prop:bisection_integer} to implement the improvement oracle $\mathrm{O}_{0,\kappa\gamma}$.
We need three assumptions: first,
$S\subseteq [a,b]\times\R^d$ with $a<b$; second, $f$ has a finite spread on the feasible set; and third
we can minimize $f(x,y)$ with $(x,y)\in S$ and $x$ fixed up to an accuracy $\gamma$.
That is, we have access to a value $\tilde\varphi(x)\in[\varphi(x),\varphi(x)+\gamma]$ with $\varphi(x):=\min\{f(x,y):(x,y)\in S\}$ being convex. 

Given a feasible query point $(x,y)\in[a,b]\times\R^d$, we
can determine correctly that there is no point $(\hat x,\bar
y) \in ((t_0 + \tau \Z) \cap [0,1]) \times \R^d$
for which $f(\hat x,\bar y)\leq f(x,y)$, provided that 
the output $\hat x$ of our approximate bisection method for integers given in
Proposition \ref{prop:bisection_integer} satisfies $\tilde
\varphi(\hat x)-\kappa\gamma>f(x,y)$.
Otherwise, we can
determine a point $(\hat x,\bar y)$ for which $f(\hat x,\bar
y)\leq f(x,y)+\kappa\gamma$. Note that this oracle cannot
report \textbf{a} and \textbf{b}
simultaneously.$\hfill\diamond$

%
%
\end{remark}

\subsection{Mixed-integer convex problems with two integer variables}\label{sec:MIP2D}

We could use the Mirror-Descent Method in Algorithm \ref{algo:mirror} to solve the generic problem \eqref{eq:general_MIP_f} when $n=2$ with $z\mapsto \frac{1}{2}||z||^2_2$ as function $V$, so that $\sigma = 1$ and $M = \frac{1}{2}\diam(S)^2$, where $\diam(S) = \max\{||z-z'||_2:z,z'\in S\}$. According to \eqref{eq:complexity_mirror}, the worst-case number of iterations is bounded by a multiple of $L\sqrt{M/\sigma} = \OO(L\,\diam(S))$, where $L$ is the Lipschitz constant of $f$. As $V_f\leq L\,\diam(S)$, the resulting algorithm would have a worst-case complexity of $\Omega(V_f)$.

We improve this straightforward approach with a variant of Algorithm \ref{algo:2d}, whose complexity is polynomial in $\ln(V_f)$. This variant takes into account the fact that we do not have access to exact values of the partial minimization function $\phi$ defined in the preamble of this section. 

\begin{proposition}\label{prop:2dMIP}
Suppose that we can determine, for every $x \in \R^n$
with $g(x)\leq 0$, a point $y_x\in\R^d$ satisfying
$f(x,y_x)-\gamma\leq\min\{f(x,y):(x,y)\in S\}$.
Then we can implement the oracle
$\mathrm{O}_{0,\kappa\gamma}$ such that for every
$(x,y) \in S$ it takes a number of evaluations of $f$ that
is polynomial in $\ln(V_f/\gamma)$.
\end{proposition}

\begin{proof}
\smartqed
We adapt the algorithm described in the proof of Theorem
\ref{2dim::theorem} for the function
$\phi(x):=\min\{f(x,y):(x,y)\in S\}$, which we only know
approximately. Its available approximation is denoted by
$\tilde\phi(x):=f(x,y_x)\in[\phi(x),\phi(x)+\gamma]$.

Let $(x,y)\in S$ be the query point and let us describe the changes that the algorithm in Theorem \ref{2dim::theorem} requires.
We borrow the notation from the proof of Theorem \ref{2dim::theorem}. 

The one-dimensional integer minimization problems which arise in the course of the algorithm require the use of our approximate bisection method for integers in Proposition \ref{prop:bisection_integer}. 
This bisection procedure detects, if it exists, a point $\hat x$ on the line of interest for which $\tilde\phi(\hat x)=f(\hat x,y_{\hat x})\leq f(x,y)+ \kappa\gamma$ and we are done. Or it reports correctly that there is no integer $\hat x$ on the line of interest with $\phi(\hat x)\leq f(x,y)$.

In \textbf{Case 2}, we would need to check whether $\phi(\hat z)\leq f(x,y)$.
In view of our accuracy requirement, we only need to check $\tilde\phi(\hat  z)\leq f(x,y)+\kappa\gamma$.

We also need to verify whether the line $H$ intersects the level set $\{x\in\R^2 \;|\; \phi(x)\leq f(x,y)\}$.
%
We use the following approximate version:
\[
\textrm{``\textbf{check whether} there is a $v\in\conv\{z_0,\hat z\}$ for which $\tilde\phi(v)< f(x,y)+(\kappa-1)\gamma$''},
\]
which can be verified using Proposition \ref{prop:bisection_continuous}. If such a point $v$ exists, the convexity of $\phi$ forbids any $w\in\conv\{\hat z,z_1\}$ to satisfy $\phi(w)\leq f(x,y)$, for otherwise:
\[
 \tilde\phi(\hat z)\leq\phi(\hat z)+ \gamma \leq\max\{\phi(v),\phi(w)\}+ \gamma
 \leq\max\{\tilde\phi(v),\tilde\phi(w)\}+ \gamma <f(x,y)+\kappa\gamma,
\]
a contradiction. Now, if such a point $v$ does not exist, we perform the same test on $\conv\{\hat z,z_1\}$.
We can thereby determine correctly which side of $\hat z$ on $H$ has an empty intersection with the level set.
%
\qed
\end{proof}

Similarly as in Corollary~\ref{2dim::optimization}, we can
extend this oracle into an approximate minimization
procedure, which solves our optimization problem up to an
accuracy of $\kappa\gamma$, provided that we have at our
disposal a point $(x,y) \in S$ such that $f(x,y)$
is a lower bound on the mixed-integer optimal value.

Let us now modify our method for finding the
$k$-th best point for two-dimensional problems to problems
with two integer and $d$ continuous variables.
Here, we aim at finding --- at least approximately --- the $k$-th best fiber $\hat x^*_k\in [-B,B]^2$, so that:
\[
(\hat x^*_k,y^*_k)\in\arg\min\{f(x,y):(x,y)\in S\cap ((\Z^2\setminus\{\hat x^*_1,\ldots,\hat x^*_{k-1}\})\times\R^d)\}
\]
for a $y^*_k\in\R^d$. We set $\hat f^*_{[k]}:=f(\hat x^*_k,y^*_k)$.
The following proposition summarizes the necessary
extensions of Subsection \ref{subsection:k-th point}.

\begin{proposition}\label{prop:kth point approx}
 Let $k \geq 2$ and let $\Pi_{k-1} := \{\hat z^*_1, \ldots,
 \hat z^*_{k-1}\} \subseteq [-B,B]^2 \cap \Z^2$
be points for which 
 $\phi(\hat z^*_i)\leq \hat f^*_i + i\kappa\gamma$, $g(\hat
 z^*_i)\leq 0$ when $1\leq i<k$ and such that
 $\conv\{\Pi_{k-1}\}\cap\Z^2=\Pi_{k-1}$. 
In a number of approximate evaluations of $f$ and
$g_1,\dots,g_m$ that is polynomial in $\ln(V_f/\gamma)$ and
$k$, one can either
\begin{enumerate}[(a)]
  \item find an integral point $\hat z^*_k \in
      [-B,B]^2$ for which $\phi(\hat z^*_k)\leq \hat f^*_{[k]}
    + k\kappa\gamma$,
$g(\hat z^*_k)\leq 0$ and $\conv\{\Pi_{k-1},\hat z^*_k\}\cap\Z^2=\Pi_{k-1}\cup\{\hat z^*_k\}$, or
  \item show that there is no integral point
   $\hat z^*_k \in [-B,B]^2$  for which $g(\hat z^*_k)\leq 0$.
\end{enumerate}
\end{proposition}
\begin{proof}
\smartqed
 If $k=2$, we run Algorithm~\ref{algo:2d} applied to
 $\hat z^*_1$
 with Line \ref{algo:2d:ip} replaced by solving
 $\min\{\langle h,\hat y\rangle: \hat y\in\bar
 T_k\cap\Z^2,\; \langle h,\hat y\rangle\geq\langle h,
\hat z^\ast_1\rangle+1\}$, where $h \in \Z^2$ such that $\gcd(h_1,h_2) = 1$.
We also need to use approximate bisection methods instead of
exact ones.
Following the proof of Proposition~\ref{prop:2dMIP}, the
oracle finds, if it exists, a feasible point $\hat z^*_2$.
Either $\tilde\phi(\hat z^*_2)\leq\tilde\phi(\hat z^*_1) + \kappa\gamma\leq \hat f^*_{[1]} + 2\kappa\gamma\leq \hat f^*_{[2]} + 2\kappa\gamma$, or $\tilde\phi(\hat z^*_2)>\tilde\phi(\hat z^*_1) + \kappa\gamma$, then $\phi(\hat z^*_2)\leq\tilde\phi(\hat z^*_2)\leq\hat f^*_{[2]} + \kappa\gamma$.
Note that, if $\phi(\hat z^*_2)>\phi(\hat z^*_1) + \kappa\gamma$, we can conclude a posteriori that $z^*_1$ corresponds precisely to $f^*_{[1]}$. 
 
 
For $k\geq 3$, we can define the same triangulation as in Figure~\ref{fig:triangulation2}. Replicating the observation sketched above, we generate indeed a feasible point $\hat z^*_k$ for which $\tilde\phi(\hat z^*_k)\leq \hat f^*_{[k]} + k\kappa\gamma$.

Lemma~\ref{lem:P_k cap Z^2 = S_k} is extended as follows. Suppose that there is an integer point $\hat x$
in $\conv\{\Pi_{k-1},\hat
z^*_k\}\setminus(\Pi_{k-1}\cup\{\hat z^*_k\})$. Since
$\phi(x)\leq \tilde\phi(x)\leq\hat f^*_{[k]}+k\kappa\gamma$ and $g(x)\leq 0$ for
every $x\in\Pi_{k-1}\cup\{\hat z^*_k\}$, we have $\phi(\hat
x)\leq\hat f^*_{[k]}+k\kappa\gamma$ and $g(\hat x)\leq 0$ by
convexity. Thus, we can apply Algorithm~\ref{algo:P_k cap
  Z^2} to find a suitable point $\hat z^*_k$ in
$\conv\{\Pi_{k-1},\hat z^*_k\}$.\qed
\end{proof}

\subsection{A finite-time algorithm for mixed-integer convex
optimization}\label{sec:MIPgen}

In this subsection, we explain how to use the results of the
previous subsection in order to realize the oracle $\Ora$ for
$\alpha\ge 0$, $\delta>0$ in the general case, i.e., with
$n\geq 3$ integer and $d$ continuous variables as in \eqref{eq:general_MIP_f}.

Let $z \in S \subseteq [-B,B]^n \times \R^d$ be the query
point of the oracle.
The oracle needs to find a point
$\hat{z} \in S \cap (\Z^n \times \R^d)$ for which
$f(\hat{z}) \le (1+\alpha) f(z) + \delta$
(so as to report {\bf a}), or to certify that $f(z) < f(\hat{z})$
for every $\hat{z} \in S \cap (\Z^n \times \R^d)$
(so as to report {\bf b}). 
To design such an oracle we have at our disposal a procedure
to realize the oracle $\Ora$ for any mixed-integer convex
minimization problem of the kind \eqref{eq:general_MIP_f}
with $n=2$.
We propose a finite-time implementation of $\Ora$ with $\alpha=0$ and $\delta=\kappa\gamma$.
The main idea is to solve the $n$-dimensional case
iteratively through the fixing of integer variables.
This works as follows.
We start by solving approximately the relaxation:
\begin{equation*}
  \hat f_{12}^*:=\min\{f(x,y):(x,y) \in
  S \cap (\Z^2 \times \R^{(n-2)+d})\}
\end{equation*}
with the techniques developed in the previous subsection. If we can solve the partial minimization problems up to an accuracy of $\gamma\leq \delta/\kappa$, we obtain a point $(\hat u_1^*,\hat u_2^*,x_3^*, \dots, x_n^*,y^*) \in
S$ with $\hat u^*_1,\hat u^*_2\in\Z$ and for which:
\[
\tilde f^*_{12}:=f(\hat u_1^*,\hat u_2^*,x_3^*, \dots, x_n^*,y^*)\leq\hat f_{12}^* + \kappa\gamma
\]
As $\hat f^*_{12}$ is a lower bound on the mixed-integer optimal value $\hat f^*$, we can make our oracle output \textbf{b} if $\tilde f^*_{12}-\kappa\gamma > f(z)$. So, assume that $\tilde f^*_{12}-\kappa\gamma \leq f(z)$. 

Then we fix $\hat x_i := \hat u_i^\ast$ for $i = 1,2$ and
solve (if $k\geq 4$; if $k=3$,
the necessary modifications are straightforward)
\begin{equation*}
  \hat f_{1234}^*:=\min\{f(x,y):(x,y) \in
  S \cap ((\hat u_1^*,\hat u_2^*) \times \Z^2 \times
  \R^{(n-4)+d})\}.
\end{equation*}
We obtain a point $(\hat u_1^*,\dots,\hat u_4^*,x_5^*, \dots, x_n^*,y^*) \in
S$ with $\hat u^*_i\in\Z$ for $1\leq i\leq 4$ and for which:
\[
\tilde f^*_{1234}:=f(\hat u_1^*,\dots,\hat u_4^*,x_5^*, \dots, x_n^*,y^*)\leq\hat f_{1234}^* + \kappa\gamma\leq \hat f^*+\kappa\gamma.
\]
Now, if $\tilde f^*_{1234}-\kappa\gamma > f(z)$,
we can make our oracle output \textbf{b}. Thus, we assume that
$\tilde f^*_{1234}-\kappa\gamma \leq f(z)$ and fix
$\hat x_i := \hat u_i^\ast$ for $1 \le i \le 4$.
Iterating this procedure we arrive at the subproblem
(again, the procedure can easily be modified if $n$ is odd):
\begin{equation*}
\min\{f(x,y):(x,y) \in
  S \cap ((\hat u_1^*,\dots,\hat u_{n-2}^*) \times \Z^2 \times
  \R^d)\}.
\end{equation*}
Let $(\hat u_1^*,\dots,\hat u_n^*,y^*) \in \Z^n \times \R^d$
be an approximate optimal solution.
If we cannot interrupt the
  algorithm, i.e., if
  $f(\hat u_1^\ast,\dots,\hat u_n^\ast,y^\ast) \not \le (1+\alpha)
  f(z) + \kappa\gamma$, we replace $(\hat u_{n-3}^*,\hat u_{n-2}^*)$ by the second best point for the corresponding mixed-integer convex minimization problem. In view of Proposition~\ref{prop:kth point approx}, the accuracy that we can guarantee on the solution is only $2\kappa\gamma$, so the criterion to output \textbf{b} must be adapted accordingly. Then we proceed with the computation of $(\hat u_{n-1}^*,\hat u_n^*)$ and so on.

It is straightforward to verify that this approach results in a finite-time algorithm for the general case. In the worst case the procedure forces us to visit all integral points in $[-B,B]^n$. However, in the course of this procedure we always have a feasible solution and a lower bound at our disposal. Once the lower bound exceeds the value of a feasible solution we can stop the procedure. It is precisely the availability of both, primal and dual information, that makes us believe that the entire algorithm is typically much faster than enumerating all the integer points in $[-B,B]^n$. 

\bibliographystyle{amsalpha}
\bibliography{literature}

\end{document}